\documentclass[reqno,11pt,centertags]{amsart}
\pdfoutput=1 
\usepackage[letterpaper,margin=1.5in,bottom=1.3in]{geometry}
\usepackage{amsmath,amsthm,amscd,amssymb,latexsym,enumerate}
\usepackage{overpic}
\usepackage{booktabs}
\usepackage[tableposition=bottom,skip=10pt]{caption}
\usepackage{float}
\usepackage{tikz}
\usetikzlibrary{arrows}
\usetikzlibrary{shapes.multipart}

\usepackage{subcaption}

\usepackage[final]{hyperref}

\allowdisplaybreaks

\newcommand*{\mailto}[1]{\href{mailto:#1}{\nolinkurl{#1}}}

\makeatletter
\def\theequation{\@arabic\c@equation}

\newcommand{\bbN}{{\mathbb{N}}}
\newcommand{\bbR}{{\mathbb{R}}}

\newcommand{\bbS}{{\mathbb{S}}} 

\newcommand{\bbC}{{\mathbb{C}}}

\newcommand{\cH}{{\mathcal H}}

\newcommand{\no}{\nonumber}
\newcommand{\lb}{\label}
\newcommand{\bi}{\bibitem}
\newcommand{\f}{\frac}

\newcommand{\ol}{\overline}

\newcommand{\dott}{\,\cdot\,}

\newcommand{\dom}{\operatorname{dom}}

\renewcommand{\Re}{\operatorname{Re}}

\numberwithin{equation}{section}

\newtheorem{theorem}{Theorem}[section]
\newtheorem{lemma}[theorem]{Lemma}
\newtheorem{proposition}[theorem]{Proposition}

\newtheorem{conjecture}[theorem]{Conjecture}
\theoremstyle{definition}

\newtheorem{remark}[theorem]{Remark}

\usepackage{graphicx}
\DeclareGraphicsExtensions{.pdf}

\def\d{\displaystyle}

\begin{document}
\title[Essential Self-Adjointness of Homogeneous Polyharmonic Operators]{Essential Self-Adjointness of Strongly Singular Homogeneous Polyharmonic Operators} 

\author{Fritz Gesztesy(*)}
\thanks{(*) Corresponding author: Fritz Gesztesy, \mailto{Fritz\_Gesztesy@baylor.edu}} 
\address{Department of Mathematics,
Baylor University, Sid Richardson Bldg., 1410 S.\,4th Street,
Waco, TX 76706, USA}
\email{\mailto{Fritz\_Gesztesy@baylor.edu}}
%\email{Fritz\_Gesztesy@baylor.edu}
\urladdr{\url{https://math.artsandsciences.baylor.edu/person/fritz-gesztesy-phd}}
%\urladdr{https://math.artsandsciences.baylor.edu/person/fritz-gesztesy-phd}

\author{Markus Hunziker}
\address{Department of Mathematics,
Baylor University, Sid Richardson Bldg., 1410 S.\,4th Street,
Waco, TX 76706, USA}
\email{\mailto{Markus\_Hunziker@baylor.edu}}
%\email{Marku\_Hunziker@baylor.edu}
\urladdr{\url{https://math.artsandsciences.baylor.edu/person/markus-hunziker-phd}}
%\urladdr{https://math.artsandsciences.baylor.edu/person/markus-hunziker-phd}

%\urladdr{\url{https://www.baylor.edu/math/index.php?id=54010}}
%\urladdr{https://www.baylor.edu/math/index.php?id=54010}

%%%%%%%%%%%%%%%%%%%%%%%%%%%%%%%%%%%%%%%%%
%\dedicatory{}
\date{\today}
\thanks{Ann. H.\ Poincar\'e (to appear).}
%\thanks{.}

%%%%%%%
\begin{abstract}
We consider essential self-adjointness of strongly singular, homogeneous, polyharmonic operators of the form 
\[
T_m = \left((-\Delta)^m +c|x|^{-2m}\right)\big|_{C_0^{\infty}(\bbR^n \backslash \{0\})},  \quad m \in \bbN, \; n \in \bbN, n\geq 2, \; c \in \bbR, 
\]
in $L^2(\bbR^n; d^n x)$, with special emphasis on the biharmonic case $m=2$ and the case $m=3$. For instance, in the biharmonic case $m=2$ we prove the sharp result   
\begin{align*} 
& \text{$T_2$ is essentially self-adjoint if and only if}    \\
& \quad 
c\geq \begin{cases}
3(n+2)(6-n) & \mbox{for $2\leq n\leq 5$};\\[5pt]
{\displaystyle -\frac{(n+4)n(n-4)(n-8)}{16}}& \mbox{for $n\geq 6$}.
\end{cases}
\end{align*}
In particular, in the special (nonsingular) case $c=0$, $(-\Delta)^2\big|_{C_0^{\infty}(\bbR^n \backslash \{0\})}$ is essentially self-adjoint in $L^2(\bbR^n; d^n x)$ if and only if $n \geq 8$. 
%
%Via separation of variables, our  proof reduces to studying the  
%essential self-adjointness on the space $C_0^{\infty}((0,\infty))$ of fourth-order Euler-type 
%differential operators associated with differential expressions  of the form 
%\[
%\frac{d^4}{dr^4} + c_1\left(\frac{1}{r^2}\frac{d^2}{dr^2} +\frac{d^2}{dr^2}\frac{1}{r^2}\right) 
%+ \frac{c_2}{r^4} , \quad r \in (0,\infty), \quad (c_1, c_2)\in \bbR^2,
%\]
%in $L^2((0,\infty);dr)$. 
%
Similarly, we derive the analogous sharp essential self-adjointness result for $T_3$ (i.e., for $m=3$) for all space dimensions $n \in \bbN$, $n \geq 2$.   

Our methods generalize to homogenous polyharmonic differential operators, however, there are some nontrivial subtleties that arise. For example, the natural expectation that for $m, n \in \bbN$, $n \geq 2$, there exist $c_{m,n} \in \bbR$ such that 
$\left((-\Delta)^m +c|x|^{-2m}\right)\big|_{C_0^{\infty}(\bbR^n \backslash \{0\})}$ is essentially self-adjoint in $L^2(\bbR^n; d^n x)$ if and only if $c \geq c_{m,n}$, turns out to be false. Indeed, for $n=20$, we prove that the differential operator 
\[
\left((-\Delta)^5 +c|x|^{-10}\right)\big|_{C_0^{\infty}(\bbR^{20} \backslash \{0\})},  \quad c \in \bbR, 
\]
is essentially self-adjoint in $L^2\big( \bbR^{20}; d^{20} x\big)$ if and only if $c\in [0,\beta]\cup [\gamma,\infty)$,
where $\beta  \approx 1.0436\times 10^{10}$, and $\gamma \approx 1.8324 \times 10^{10}$ are the two real roots of a particular quartic equation with integer coefficients (see Theorem \ref{t4.4}, eq.~\eqref{4.26}). 
\end{abstract}
%%%%%%%

\maketitle
 
%%%%%%% 
\section{Introduction} \lb{s1}
%%%%%%%
%%%%%%%

Polyharmonic operators in Hilbert spaces, that is, $L^2$-realizations of differential expressions of the form\footnote{We consider powers of $-\Delta$ to ensure nonnegativity of $(-\Delta)^m\big|_{H^{2m}(\bbR^n)}$ in $L^2(\bbR^n;d^nx)$.}

\begin{equation}
(-\Delta)^m + V(x), \quad x \in \bbR^n, \; m, n \in \bbN, 
\end{equation}
especially, the biharmonic case $m=2$, have long been of interest in connection with spectral theory. As a few representative cases we mention, for instance, Floquet theory in connection with periodic $V$ and the Bethe--Sommerfeld conjecture \cite[Chs.~2--4]{Ka97}, \cite{PS01}, \cite{PS01a}, \cite{PS10}, scattering theory for relative trace class perturbations $V$ \cite{Pu99}, \cite[Sect.~3.4]{Ya10}, direct and inverse source scattering \cite{Zh25}, the asymptotic distribution of eigenvalues, in particular, Weyl asymptotics  \cite{Li16}, \cite[Ch.~1]{SV97}, bounds for the number of eigenvalue \cite{Hu18}, regions of positivity of polyharmonic Green's functions \cite{GS07}, eigenvalue dependence on domains \cite{BL13}, the Calder{\' o}n problem for polyharmonic operators \cite{SS23}, and inverse boundary problems \cite{KLU12}, \cite{KU16}, to name just a few relevant situations. 

If $V$ is strongly singular, then matters of (essential) self-adjointness become of utmost importance. Indeed,  
self-adjointness for operators in a Hilbert space is a fundamental property as, in conjunction with the spectral theorem, it permits the construction of particular functions of a self-adjoint operator such as unitary groups, semigroups, cosine and sine functions, etc., which are basic in solving standard initial value problems which are first-order (such as heat or Schr\"odinger equation) and second-order (such as the wave equation) with respect to the  time variable. In this paper we focus on essential self-adjointness of certain families of strongly singular, homogeneous polyharmonic differential operators. 

Before describing the principal topic of this paper, we briefly pause to introduce the basic underlying notions: Let $\cH$ denote a complex, separable Hilbert space. Then a linear operator $T$ in $\cH$ with domain $\dom(T) \subseteq \cH$ dense in $\cH$, that is, $\ol{\dom(T)} = \cH$, is called {\it symmetric}, if $T$ is a restriction of its {\it adjoint} $T^*$ (equivalently, $T^*$ is an extension of $T$), in symbols, $T \subseteq T^*$. Explicitly, this entails 
\begin{equation}
\dom(T) \subseteq \dom(T^*) \, \text{ and } \, Tf = T^*f, \quad f \in \dom(T). 
\end{equation} 
$T$ is called {\it self-adjoint} if $T=T^*$. Moreover, $T$ is called {\it essentially self-adjoint}, if the {\it closure} of $T$, denoted by $\ol T$, and given by the double adjoint $(T^*)^*$ (i.e., $\ol T = (T^*)^*$), is self-adjoint. Since taking the closure of a (closable) operator basically follows a standard procedure involving Cauchy sequences, proving that an operator is essentially self-adjoint is typically the major step toward proving self-adjointness after an operator is recognized to be symmetric. 

In the following, the underlying Hilbert space $\cH$ will be the concrete Lebesgue space $L^2(\bbR^n;d^nx)$, $n \in \bbN$, $n \geq 2$, with $d^n x$ denoting the Lebesgue measure on $\bbR^n$, and, in the particular biharmonic case,  the operator $T$ in question will be the strongly singular, homogeneous differential operator of the type 
\begin{equation}
T_2(c) = \left((-\Delta)^2 +c|x|^{-4}\right)\big|_{C_0^{\infty}(\bbR^n \backslash \{0\})}, \quad c \in \bbR,  
\end{equation}
in $L^2(\bbR^n;d^nx)$. In one of our principal results we will prove the sharp statement 
\begin{align} 
\begin{split} 
& \text{$T_2(c)=\left((-\Delta)^2 +c|x|^{-4}\right)\big|_{C_0^{\infty}(\bbR^n \backslash \{0\})}$ is essentially self-adjoint in $L^2(\bbR^n;d^nx)$}     \lb{1.3}\\
& \quad \text{if and only if } \,  
c \geq \begin{cases}
3(n+2)(6-n) & \mbox{for $2\leq n\leq 5$},   \\[5pt]
{\displaystyle -\frac{(n+4)n(n-4)(n-8)}{16}}& \mbox{for $n\geq 6$}.
\end{cases}
\end{split} 
\end{align} 
In particular, in the special case $c=0$, 
\begin{align}
\begin{split} 
& \text{$(-\Delta)^2\big|_{C_0^{\infty}(\bbR^n \backslash \{0\})}$ is essentially self-adjoint in 
$L^2(\bbR^n; d^n x)$}       \lb{1.4} \\ 
& \quad \text{if and only if $n \geq 8$}. 
\end{split}
\end{align} 

Put differently, \eqref{1.3} and \eqref{1.4} assert that $C_0^{\infty}(\bbR^n \backslash \{0\})$ is an operator core for the closure, $\ol{T_2(c)}$, of $T_2(c)$, where $\ol{T_2(c)}$ represents the natural self-adjoint operator associated with $T_2(c)$ in $L^2(\bbR^n;d^nx)$. 

Via separation of variables, our proof of \eqref{1.3} reduces to studying the  
essential self-adjointness on the space $C_0^{\infty}((0,\infty))$ of fourth-order Euler-type differential operators associated with differential expressions  of the form 
\begin{equation}
\tau_2(c_1,c_2) = \frac{d^4}{dr^4} + c_1\left(\frac{1}{r^2}\frac{d^2}{dr^2} +\frac{d^2}{dr^2}\frac{1}{r^2}\right) + \frac{c_2}{r^4} , \quad r \in (0,\infty), \quad (c_1, c_2)\in \bbR^2,
\end{equation} 
in $L^2((0,\infty);dr)$. We will prove that 
\begin{align} \lb{1.5}
\begin{split} 
& \text{$\tau_2(c_1,c_2)|_{C_0^{\infty}((0,\infty))}$ is essentially self-adjoint in $L^2((0,\infty);dr)$} \\
& \quad \text{if and only if }
c_2\geq 
\begin{cases}
45 + 12c_1 + c_1^2 & \mbox{for $c_1 \geq   -11/4$},   \\[10pt]
{\d -\frac{105}{16}-\frac{19}{2} c_1}& \mbox{for $c_1 <   -11/4$}.
\end{cases}
\end{split} 
\end{align} 
This, in turn, is a consequence of the fact that 
\begin{align}
\begin{split} 
& \text{\it  $\tau_2(c_1,c_2)|_{C_0^{\infty}((0,\infty))}$ is essentially self-adjoint if and only if } \\
&  \quad \text{\it exactly two roots  of $D_2(c_1,c_2;\dott)$ have real part $\leq -1/2$}\\
& \quad \text{\it and the two remaining roots have real part $>-1/2$.}
\end{split} 
\end{align}
Here $D_2(c_1,c_1;\dott)$ is the quartic polynomial given by
\begin{equation}
D_2(c_1,c_1; z)=z(z-1)(z-2)(z-3) +c_1\left[z(z-1)+(z-2)(z-3)\right] +c_2. 
\end{equation}

We were not able to find a result of the type \eqref{1.3} for the homogeneous biharmonic operator $T_2(c)$ in the literature. In stark contrast to this, the analogous result for $T_1(c)$, given by 
\begin{equation}
T_1(c) = \left(-\Delta +c|x|^{-2}\right)\big|_{C_0^{\infty}(\bbR^n \backslash \{0\})}, \quad c \in \bbR,  
\end{equation}
in $L^2(\bbR^n;d^nx)$, is well-known and considered a classical result in connection with strongly singular Schr\"odinger operators. Indeed, one obtains (see also the end of Sect.~1 in \cite{KW72}), 
\begin{align} 
\begin{split} 
& \text{$T_1(c) = \left(-\Delta +c|x|^{-2}\right)\big|_{C_0^{\infty}(\bbR^n \backslash \{0\})}$ is essentially self-adjoint in $L^2(\bbR^n;d^nx)$}     \lb{1.6} \\
& \quad \text{if and only if } \,  
c \geq - \f{n(n-4)}{4}. 
\end{split} 
\end{align} 
In particular, in the special case $c=0$, 
\begin{align}
\begin{split} 
& \text{$(-\Delta)\big|_{C_0^{\infty}(\bbR^n \backslash \{0\})}$ is essentially self-adjoint in 
$L^2(\bbR^n; d^n x)$}     \lb{1.7} \\
& \quad \text{if and only if $n \geq 4$.} 
\end{split} 
\end{align} 
Via separation of variables, the proof of \eqref{1.6} (resp., \eqref{1.7}) reduces to studying the  
essential self-adjointness on the space $C_0^{\infty}((0,\infty))$ of the classical second-order Bessel-type differential operators associated with differential expressions  of the form 
\begin{equation}
\tau_1(c_1) = - \frac{d^2}{dr^2} + \frac{c_1}{r^2} , \quad r \in (0,\infty), \quad c_1 \in \bbR,     \lb{1.9} 
\end{equation} 
in $L^2((0,\infty);dr)$. It is well-known, in fact, a classical result, that 
\begin{align}
\begin{split}
& \text{$\tau_1(c_1)|_{C_0^{\infty}((0,\infty))}$ is essentially self-adjoint in $L^2((0,\infty);dr)$} \\
& \quad \text{if and only if $c_1 \geq 3/4$}.
\end{split}
\end{align}  

For relevant references in this context see, for instance, \cite{DF23}, \cite{DG21}, \cite{DG22}, \cite{DR17},  \cite[p.~33--35]{Fa75}, \cite{GPS24}, \cite{Gr74}, \cite{Jo64}, \cite{KSWW75}, \cite{KW72}, \cite{KW73}, \cite[Theorem~X.11, Example 4 on p.~172, Theorem~X.30]{RS75}, \cite{Sc72}, and \cite{Si73}. 

While a systematic discussion of higher-order situations is met with obstacles as the relative position of real values of zeros of even-order polynomials play a crucial role in the analysis (see,  Lemmas \ref{l2.1}, \ref{l3.1}, and \ref{l4.2} and their use in the proofs of Theorems \ref{t2.2}, \ref{t3.2}, and \ref{t4.2}) we now record the next case in line: Consider 
\begin{equation}
T_3(c) = \left((-\Delta)^3 +c|x|^{-6}\right)\big|_{C_0^{\infty}(\bbR^n \backslash \{0\})},  \quad c \in \bbR,   
\end{equation}
in $L^2(\bbR^n;d^nx)$. Then 
\begin{align}
&T_3(c) = \left((-\Delta)^3 +c|x|^{-6}\right)\big|_{C_0^{\infty}(\bbR^n \backslash \{0\})}\ \mbox{is essentially self-adjoint  in $L^2(\bbR^n; d^n x)$}     \no \\
& \quad \text{if and only if}     \lb{1.11} \\ 
& \;\; c\geq \begin{cases}
\displaystyle \frac{64}{27} \left(7112+504n -126 n^2 +(236+12n-3 n^2)\sqrt{964+60n-15 n^2}\right)  \\[5pt]
\hspace*{8.95cm} \mbox{for $2\leq n\leq 9$},  \\[5pt]
\displaystyle-\frac{(n+8)(n+4)n(n-4)(n-8)(n-12)}{64} \hspace*{2.15cm} \mbox{for $n\geq 10$}.\\
\end{cases}  \no
\end{align}
Once more, we are not aware of such a result in the literature. In particular, in the special case $c=0$, 

\begin{align} 
\begin{split} 
& \text{$(-\Delta)^3\big|_{C_0^{\infty}(\bbR^n \backslash \{0\})}$ is essentially self-adjoint in $L^2(\bbR^n; d^n x)$}   \\
& \quad \text{if and only if $n \geq 12$.} 
\end{split}
\end{align}

At this point it may seem natural to ask the following question: 
For general $m,n\in \bbN$, $n\geq 2$, 
 \begin{align}
\begin{split} 
& \text{\it does there exist $c_{m,n} \in \bbR$ such that} \\ 
& \quad T_m(c) = \big((-\Delta)^m +c|x|^{-2m}\big)\big|_{C_0^{\infty}(\bbR^n \backslash \{0\})},  \quad c \in \bbR, \\
& \quad \text{\it is essentially self-adjoint in $L^2(\bbR^n;d^nx)$ if and only if $c \geq c_{m,n}$?}
\end{split} \lb{1.12}
\end{align}
It may come as a surprise, however, that the answer to question \eqref{1.12} is negative in general. Indeed, we will prove the following fact: Consider $n=20$ and 
\begin{equation}
T_5(c) = \left((-\Delta)^5 +c|x|^{-10}\right)\big|_{C_0^{\infty}\big(\bbR^{20} \big\backslash \{0\}\big)}, \quad c \in \bbR,   
\end{equation}
in $L^2\big(\bbR^{20};d^{20}x\big)$. Then 
\begin{align}
& T_5(c) = \left((-\Delta)^5 +c|x|^{-10}\right)\big|_{C_0^{\infty}(\bbR^{20} \backslash \{0\})}\ \mbox{is essentially self-adjoint  in $L^2\big(\bbR^{20}; d^{20}x\big)$}   \no \\
& \quad \text{if and only if $c \in [0,\beta] \cup [\gamma,\infty)$},     \lb{1.16} 
\end{align}
where $\beta  \approx 1.0436\times 10^{10}$, and $\gamma \approx 1.8324 \times 10^{10}$ are the two real roots of the quartic equation
\begin{align}
&3125 z^4-83914629120000 z^3+429438995162964368031744 z^2     \no \\
&\quad +1045471534388841527438982355353600 z\\  
&\quad +629847004905001626921946285352115240960000 =0.    \no
\end{align}

In particular, for $n=20$, $T_5(c)$, for $c \in (\beta,\gamma)$, displays what one could call an ``island'' of non-essential self-adjointness in $L^2\big(\bbR^{20}; d^{20}x\big)$. 

While essential self-adjointness of the operators $\left((-\Delta)^m +c|x|^{-2m}\right)\big|_{C_0^{\infty}(\bbR^n \backslash \{0\})}$ in $L^2(\bbR^n; d^nx))$, $m \in \bbN$, depends in a subtle manner on the properties of zeros of certain polynomials, it might be instructive to compare this with the corresponding property of boundedness from below, in fact, with nonnegativity. The higher-order Hardy--Rellich-type inequalities derived in \cite{Ya99} (see also \cite[p.~39--40]{BEL15} and \cite{He77}) are of the form 
\begin{align}
\begin{split} 
& \int_{\bbR^n} d^n x \, \big|\big((-\Delta)^{m/2} f\big)(x)\big|^2 \geq C_{m,n} \int_{\bbR^n} d^n x \, |x|^{-2m} |f(x)|^2,     \lb{1.21} \\
& f \in C_0^{\infty}(\bbR^n\backslash\{0\}), \; m < n/2, \, \text{ or } \, m > n/2 \, \text{ and } \, 
(m - (n/2)) \notin \bbN_0, 
\end{split}
\end{align}
with optimal constants
\begin{equation}
C_{m,n} = 2^{2m} (n-2m)^{-2} (n-2m +4)^2 \cdots (n+2m-4)^{-2}, \quad m, n \in \bbN.   \lb{1.22} 
\end{equation} 
(Inequality \eqref{1.21} cannot hold for any constant $C \in [0,\infty)$ if $(m - (n/2)) \in \bbN_0$, see \cite{Ya99}.) Thus, 
\begin{align}
\begin{split} 
& \left((-\Delta)^m +c|x|^{-2m}\right)\big|_{C_0^{\infty}(\bbR^n \backslash \{0\})} \geq 0 \, \text{ in $L^2(\bbR^n; d^nx))$}    \lb{1.23} \\
& \quad \text{if and only if $c \geq - C_{m,n}$,}
\end{split} 
\end{align}
given the restrictions on $m,n \in \bbN$ described in \eqref{1.21}. Thus, in sharp contrast to the essential self-adjointness property, the boundedness from below (in fact, nonnegativity) property in \eqref{1.23} has an explicit answer for all permissible $m,n \in \bbN$.

In Section \ref{s2} we analyze $\tau_2(c_1,c_2)$ in great detail and establish \eqref{1.5}. The fourth-order fact \eqref{1.3} is the principal result, Theorem \ref{t3.2}, of Section \ref{s3}. The sixth-order fact \eqref{1.11} 
is proved in Theorem \ref{t4.2}, and the tenth-order fact \eqref{1.16} is derived in Theorem \ref{t4.4}, the two principal results of Section \ref{s4}. Finally, in Appendix \ref{sA}, we describe the fundamental system of solutions  of the fourth-order ordinary differential equation $\tau_2(c_1,c_2) y = \lambda y$, $\lambda \in \bbC$ (i.e., the generalized eigenvalue equation), in terms of the generalized hypergeometric function $\,_{0} F_3\left(\!\begin{array}{c}\\ { a, b, c} \end{array} \bigg\vert\,  z \right)$ and Meijer's $G$-function $G_{0,4}^{2,0}\left(\!\begin{array}{c}\\
{\alpha, \beta,\gamma,\delta}\end{array} \bigg\vert\, z \right)$. 
 
We employ the convenient abbreviation $\bbN_0 = \bbN \cup \{0\}$.

%%%%%%%
%%%%%%% 
\section{A Two-Parameter Family of Fourth-Order Euler-Type Differential Operators on the Half-Line} \lb{s2}
%%%%%%%
%%%%%%%

As a preparation for our principal Sections \ref{s3} and \ref{s4}, we now consider essential self-adjointness on the space $C_0^{\infty}((0,\infty))$ of $L^2$-realizations associated with differential expressions  of the type
\begin{equation}
\tau_2(c_1,c_2)= \frac{d^4}{dr^4} + c_1\left(\frac{1}{r^2}\frac{d^2}{dr^2} +\frac{d^2}{dr^2}\frac{1}{r^2}\right) +  \frac{c_2}{r^4} , \quad r \in (0,\infty), \quad (c_1, c_2)\in \bbR^2,
\end{equation}
in $L^2((0,\infty);dr)$. We note that for $r>0$, $(c_1, c_2)\in \bbR^2$, and $z \in \bbC$,   
\begin{equation}
\tau_2(c_1,c_2)r^z = D_2(c_1,c_1;z)r^{z-4},
\end{equation}
where $D_2(c_1,c_1;\dott)$ is the quartic polynomial given by
\begin{equation}
D_2(c_1,c_1; z)=z(z-1)(z-2)(z-3) +c_1\left[z(z-1)+(z-2)(z-3)\right] +c_2. 
\end{equation}

%%%%%%%
\begin{lemma}\lb{l2.1}
Let $(c_1,c_2)\in \bbR^2$. If the polynomial $D_2(c_1,c_2;\dott)$ has a root with real part equal to $-1/2$, then
\begin{equation}\lb{E:det=0}
(45 + 12 c_1 + c_1^2 - c_2) \bigg(\frac{105}{16} + \frac{19}{2} c_1 + c_2\bigg)=0.
\end{equation}
\end{lemma}
%%%%%%%
\begin{proof} Consider the  polynomial $\widetilde{D}_2(c_1,c_2;\dott)$ given by 
\begin{equation}
\widetilde{D}_2(c_1,c_2; z)= D_2(c_1,c_1; z-1/2), \quad z\in \bbC;
\end{equation}
one notes that $\widetilde{D}_2(c_1,c_2;\dott)$ has real coefficients. Explicitly, for $z\in \bbC$, one has
\begin{align}
\begin{split} 
&\widetilde{D}_2(c_1,c_2; z)\\
&\quad =z^4 - 8z^3+ \left(  \frac{43}{2} +2 c_1\right)z^2 +\left(-22-8 c_1\right)z+ \left(\frac{105}{16}+\frac{19}{2}c_1+c_2\right),  
\end{split} 
\end{align}
and hence the  Hurwitz matrix (see, \cite[\S~XV.6]{Ga59}) associated  with $\widetilde{D}_2(c_1,c_2;\dott)$ is of the form 
\begin{equation}
H_2(c_1,c_2)=
\left(
\begin{array}{cccc}
 -8 & -22 -8 c_1 & 0 & 0 \\[5pt]
 1 & \frac{43}{2} +2 c_1 & \frac{105}{16} +\frac{19}{2}c_1+c_2& 0 \\[5pt]
 0 & -8 & -22-8 c_1 & 0 \\[3pt]
 0 & 1 & \frac{43}{2} +2 c_1& \frac{105}{16}+\frac{19}{2}c_1+c_2 \\[5pt]
\end{array}
\right).
\end{equation}
The determinant of $H_2(c_1,c_2)$ can be found by expansion along the last column:
\begin{align} 
\begin{split} 
\det (H_2(c_1,c_2)) &= 64(45 + 12 c_1 + c_1^2 - c_2) \bigg(\frac{105}{16} + \frac{19}{2} c_1 + c_2\bigg)    \\
& = 4 (45 + 12 c_1 + c_1^2 - c_2) (105 + 152 c_1 + 16 c_2).      \lb{E:det}
\end{split} 
\end{align}
Suppose $D_2(c_1,c_2;\dott)$ has a root with with real part equal to $-1/2$. Then 
$\widetilde{D}_2(c_1,c_2;\dott)$ has a root $\alpha\in \bbC$  with $\Re(\alpha)=(\alpha+\overline{\alpha})/2=0$.
Since $\widetilde{D}_2(c_1,c_2;\dott)$ has real coefficients,
$\overline{\alpha}$ is also a root of $\widetilde{D}_2(c_1,c_2;\dott)$. It now follows from Orlando's formula (see \cite[\S~XV.7]{Ga59}) that $\det (H_2(c_1,c_2)) = 0$. 
By \eqref{E:det}, we conclue that \eqref{E:det=0} is satisfied.
\end{proof}
%%%%%%%

\begin{figure}[h]
\includegraphics[width=0.45\textwidth]{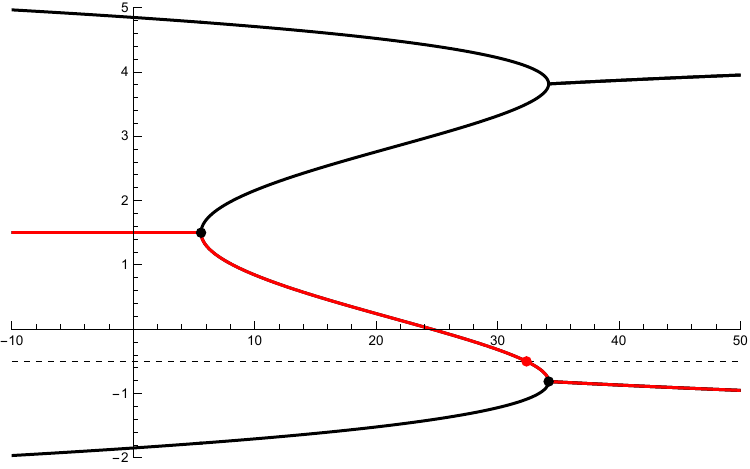}\qquad
\includegraphics[width=0.45\textwidth]{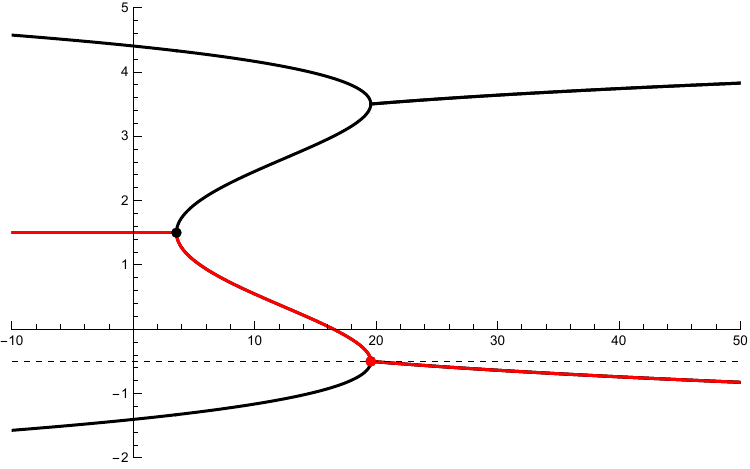}\\[1pc]
\includegraphics[width=0.45\textwidth]{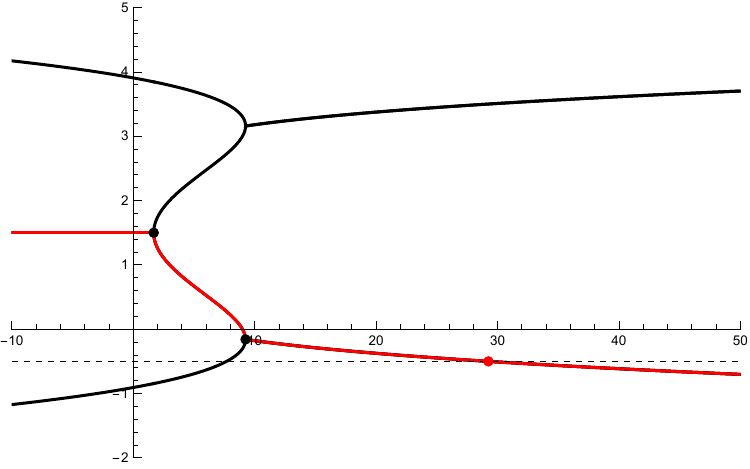}\qquad
\includegraphics[width=0.45\textwidth]{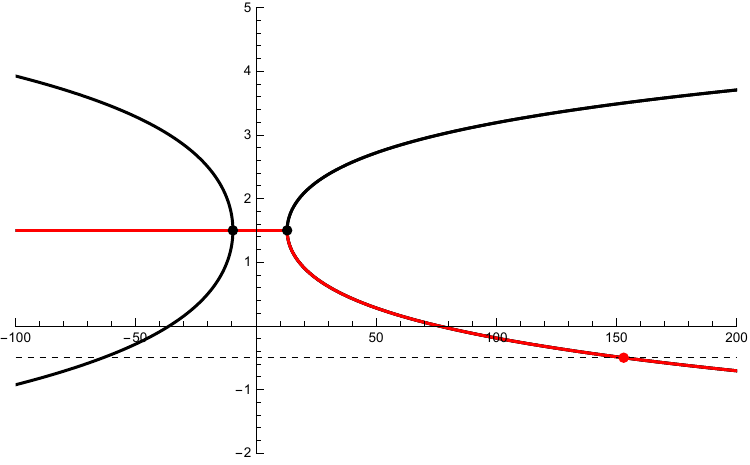}
\caption{Graphs of the functions $\Re(\alpha_j(c_1,\dott))$, $1\leq j\leq 4$, for  $c_1<-11/4$ (top left),  $c_1=-11/4$ (top right), $-11/4<c_1<5/4$ (bottom left),
and $5/4<c_1$ (bottom right)
respectively. The graph of $\Re(\alpha_2(c_1,\dott))$ is shown in red. 
The dashed horizontal line  is the line $y=-1/2$.  The red dot is the point $(t_0, -1/2)$, where the  graph of $\Re(\alpha_2(c_1,\dott))$ crosses the dashed line.
The black dots are the points $(a,3/2)$ and $(b,\Re(\alpha_2(c_1,b)))$, respectively. (For $c_1=-11/4$, we have  $t_0=b$ and hence 
the red dot and the black dot corresponding to the point $(b,\Re(\alpha_2(c_1,b)))$ coincide.)}
\lb{F:t0}
\end{figure}

%%%%%%%
\begin{theorem} \lb{t2.2}
Let $(c_1,c_2)\in \bbR^2$. Then 
\begin{align} \lb{2.9}
\begin{split} 
& \text{$\tau_2(c_1,c_2)|_{C_0^{\infty}((0,\infty))}$ is essentially self-adjoint in $L^2((0,\infty);dr)$} \\
& \quad \text{if and only if }
c_2\geq 
\begin{cases}
45 + 12c_1 + c_1^2 & \mbox{for $c_1 \geq   -11/4$},   \\[10pt]
- (105 + 152 c_1)/16 & \mbox{for $c_1 <   -11/4$}.
\end{cases}
\end{split} 
\end{align} 
\end{theorem}
%%%%%%%
\begin{proof}
Let $(c_1,c_2)\in \bbR^2$.  By our previous work \cite[Remark~2.1]{GHT23},
\begin{align}
\begin{split} 
& \text{\it  $\tau_2(c_1,c_2)|_{C_0^{\infty}((0,\infty))}$ is essentially self-adjoint if and only if } \\
&  \quad \text{\it exactly two roots  of $D_2(c_1,c_2;\dott)$ have real part $\leq -1/2$}\\
& \quad \text{\it and the two remaining roots have real part $>-1/2$.}
\end{split} 
\end{align}
The roots of $D_2(c_1,c_2;\dott)$ are explicitly given as follows,  
\begin{align}
\alpha_1(c_1,c_2)&= \frac{3}{2}-\frac{1}{2}\sqrt{5-4 c_1+4 \sqrt{1 - 4 c_1 + c_1^2 - c_2}}, \lb{2.12} \\
\alpha_2(c_1,c_2)&= \frac{3}{2}-\frac{1}{2}\sqrt{5-4 c_1-4 \sqrt{1 - 4 c_1 + c_1^2 - c_2}} ,\\
\alpha_3(c_1,c_2)&= \frac{3}{2}+\frac{1}{2}\sqrt{5-4 c_1-4 \sqrt{1 - 4 c_1 + c_1^2 - c_2}} ,\\
\alpha_4(c_1,c_2)&=\frac{3}{2}+\frac{1}{2}\sqrt{5-4 c_1+4 \sqrt{1 - 4 c_1 + c_1^2 - c_2}} . \lb{2.15}
\end{align}
Here $\sqrt{\dott}$ denotes the principal branch of the complex square root function so that 
$\sqrt{z} = \sqrt{r}e^{i \varphi/2}$ for  $ z = r e^{i \varphi}$ with $r>0$ and  $-\pi<\varphi \leq \pi$. 

It is not difficult to verify that 
\begin{equation}
\Re(\alpha_1(c_1,c_2)) \leq \Re(\alpha_2(c_1,c_2)) \leq 3/2.
\end{equation}
Since 
\begin{equation} 
\alpha_1(c_1,c_2)+\alpha_4(c_1,c_2)=\alpha_2(c_1,c_2)+\alpha_3(c_1,c_2)=3, 
\end{equation} 
it then follows that $D_2(c_1,c_2;\dott)$ has exactly two roots with  real part $\leq -1/2$
if and only if $\Re(\alpha_2(c_1,c_2)) \leq -1/2$.  

We will now study the behavior of the continuous function
\begin{equation}
\Re(\alpha_2(c_1,\,\cdot \,)): \begin{cases} \bbR\to \bbR, \\
t\mapsto \Re(\alpha_2(c_1,t )).
\end{cases} 
\end{equation}
It will be helpful to introduce two  real constants
\begin{equation}
a=- (9 + 24 c_1)/16 \quad \mbox{and} \quad b=1-4c_1+c_1^2 .
\end{equation}
One observes that  $a\leq  b$ with equality if and only if $c_1=5/4$.
For $t\in \bbR$,  one obtains 
\begin{equation}\lb{E:const=3/2}
\Re(\alpha_2(c_1,t )) = 3/2  
\quad \mbox{if and only if}\quad  t \leq 
\begin{cases}
b  & \mbox{for $c_1 \geq   5/4$},\\ 
a & \mbox{for $c_1 < 5/4$}.
\end{cases}
\end{equation}
%$\Re(\alpha_2(c_1,t ))\ll-1/2$ for $t\gg 0$.
%Therefore,  $\Re(\alpha_2(c_1,t_0 ))=-1/2$ for some $t_0\in (a,\infty)$. 
We also find that the denominator of the derivative 
\begin{equation}\lb{E:derivative}
\frac{d}{dt}[ \alpha_2(c_1,t)]=\frac{-1}{2 \sqrt{5-4c_1-4 \sqrt{1-4c_1+c_1^2-t}}\, \sqrt{1-4c_1+c_1^2-t}}
\end{equation}
vanishes if and only if  $t=a$ or  $t=b$. 
Analyzing the real part of the derivative  \eqref{E:derivative}, one shows that $\Re(\alpha_2(c_1,\,\cdot \,))$ is strictly decreasing on $(b,\infty)$ for $c_1 \geq   5/4$,
and strictly decreasing on both $(a,b)$ and $(b,\infty)$ for $c_1<   5/4$.
Therefore, by \eqref{E:const=3/2} and since $\Re(\alpha_2(c_1,t )) < -1/2$ for $t \gg b$, 

\begin{equation}\label{E:t0}
\mbox{there exists a unique $t_0\in \bbR$ such that  
$\Re(\alpha_2(c_1,t_0)) = -1/2$.}
\end{equation}

By Lemma~\ref{l2.1}, we must have 
\begin{equation}\label{E:or}
t_0 = 45 + 12c_1 + c_1^2 \quad \text{or} \quad t_0 = -(105 + 152c_1)/16.
\end{equation}
We have 
$45 + 12c_1 + c_1^2 = -(105 + 152c_1)/16$
if and only if $c_1 = -11/4$ or $c_1 = -75/4$.
Furthermore, if $c_2 = 45 + 12c_1 + c_1^2$, then the roots of $D(c_1, c_2;\,\dott)$ are
\begin{equation}\label{E:special roots}
-\frac{1}{2} \pm \frac{1}{2}\sqrt{-4c_1 - 11}
\quad \text{and} \quad
\frac{7}{2} \pm \frac{1}{2}\sqrt{-4c_1 - 11}.
\end{equation}
Hence, for $c_1 \geq  -11/4$ and $c_2 = 45 + 12c_1 + c_1^2$, it follows from \eqref{E:special roots} that 
$\Re(\alpha_2(c_1, c_2)) = -1/2$. 
By \eqref{E:t0}, we therefore obtain
\begin{equation}
t_0 = 45 + 12c_1 + c_1^2 
\quad \text{for} \quad c_1 \geq -11/4.
\end{equation}
For $c_1 < -11/4$ and $c_2 = 45 + 12c_1 + c_1^2$, 
\eqref{E:special roots} shows that 
$\Re(\alpha_2(c_1, c_2)) \ne -1/2$ 
unless $c_1 = -75/4$.
Combining \eqref{E:t0} and \eqref{E:or}, and noting that 
$45 + 12c_1 + c_1^2 = -(105 + 152c_1)/16$ 
for $c_1 = -75/4$, we find that
\begin{equation}
t_0 = -(105 + 152c_1)/16
\quad \text{for} \quad c_1 <  -11/4.
\end{equation}
This concludes the proof of the theorem.

We can say a bit more about the position of $t_0$ relative to  $a$ and $b$. Since 
\begin{equation}
\Re(\alpha_2(c_1,b ))
\begin{cases}
>-1/2  & \mbox{for $c_1> -11/4$},\\
=-1/2 & \mbox{for $c_1= -11/4$},\\
<-1/2 & \mbox{for $c_1<-11/4$},
\end{cases}
\end{equation}
it follows  that $t_0 \in (b,\infty)$ for  $c_1> -11/4$, $t_0=b$  for $c_1= -11/4$, and $t_0\in (a,b)$ for $c_1<-11/4$.
This is illustrated in Figure~\ref{F:t0}. 
\end{proof}
%%%%%%%

%%%%%%%
%%%%%%%
\section{The Strongly Singular Biharmonic Homogeneous Differential Operator 
$\big((- \Delta)^2 +c|x|^{-4}\big)\big|_{C_0^{\infty}(\bbR^n \backslash \{0\})}$, $n \geq 2$} \lb{s3}
%%%%%%%
%%%%%%%

In the following, 
\begin{equation}
\Delta= \sum_{j=1}^n \frac{\partial^2}{\partial x_j^2}, \quad x=(x_1,\ldots,x_n) \in \bbR^n,\; n\in \bbN,
\end{equation}
represents the  Laplacian on $\bbR^n$.
We consider essential self-adjointness on the space $C_0^{\infty}(\bbR^n \backslash \{0\})$
of differential operators associated with the homogeneous biharmonic differential expressions  of the type
\begin{equation}
(-\Delta)^2 +c|x|^{-4}, \quad x\in \bbR^n \backslash \{0\}, \quad n \in \bbN, \; n\geq 2,\quad  c\in \bbR, 
\end{equation}
in $L^2(\bbR^n;d^nx)$. 

In addition, we introduce 
\begin{align} \lb{E:tau_nl}
\begin{split} 
\tau_{2,n,\ell}(c)= \left[- \frac{d^2}{dr^2} + \frac{(n + 2\ell - 1)(n+2\ell-3)}{4r^2}\right]^2+\frac{c}{r^4},&    \\
r \in (0,\infty), \quad n \in \bbN, \; n \geq 2, \quad \ell\in \bbN_0, \quad c\in \bbR.&
\end{split} 
\end{align}
We note that for all $n \geq 2$, $\ell\in \bbN_0$, $c\in \bbR$, one has 
\begin{equation}\label{E:tau_n+2l,0}
\tau_{2,n,\ell}(c) = \tau_{2,n+2\ell,0}(c).
\end{equation}

 %%%%%%%%
\begin{lemma} \lb{l3.1}
Let $n \in \bbN$, $n\geq 2$. Then 
\begin{align}
& \left((-\Delta)^2 +c|x|^{-4}\right)\big|_{C_0^{\infty}(\bbR^n \backslash \{0\})}\ \mbox{is essentially self-adjoint  in $L^2(\bbR^n; d^n x)$}     \no \\
& \quad \text{if and only if}     \\
& \, \tau_{2,n,\ell}(c)\big|_{C_0^{\infty}((0,\infty))} \   \mbox{is essentially self-adjoint in $L^2((0,\infty); dr)$ for all $\ell\in \bbN_0$.}    \no
\end{align}
\end{lemma}
%%%%%%%%
\begin{proof}  We have a smooth diffeomorphism $\Phi:(0,\infty)\times \mathbb{S}^{n-1}\rightarrow \bbR^n \backslash \{0\}$
given by $\Phi(r ,\theta)=r\theta$. This  separation of variables  then leads to a canonical  decomposition
\begin{align}
\begin{split} 
L^2(\bbR^n;d^nx) &= L^2\big((0,\infty); r^{n-1}dr\big) \otimes L^2\big(\bbS^{n-1};d^{n-1} \omega\big)   \\
&= L^2\big((0,\infty); r^{n-1}dr\big) \otimes {\bigoplus}_{\ell \in \bbN_0} \mathcal{H}_\ell\big(\mathbb{S}^{n-1}\big).
\end{split} 
\end{align}
Here $\mathcal{H}_\ell(\mathbb{S}^{n-1})$ is the space of spherical harmonics of degree $\ell$.
Explicitly, $\mathcal{H}_\ell(\mathbb{S}^{n-1})$ consists of the linear span of all functions $Y_\ell(\dott)\in C^\infty(\mathbb{S}^{n-1})$
such that 
\begin{equation}
-\Delta_{\mathbb{S}^{n-1}}Y_\ell(\theta) = \ell(\ell+n-2)Y_\ell(\theta), \quad \ell \in \bbN_0, 
\end{equation} 
where $\Delta_{\mathbb{S}^{n-1}}$ denotes the Laplace--Beltrami operator in $L^2\big(\mathbb{S}^{n-1}; d^{n-1}\omega\big)$, and $d^{n-1}\omega$ represents  the usual  surface measure on $\mathbb{S}^{n-1}$.
By the well-known formula for the radial part of the Laplacian, for $f \in C_0^{\infty}((0,\infty))$ 
and $Y_\ell \in \mathcal{H}_\ell(\mathbb{S}^{n-1})$, one gets 
\begin{equation}
- \Delta\bigg(f(r)Y_\ell(\theta)\bigg)=  
\left( \left[- \frac{d^2}{dr^2} - \frac{n-1}{r}\frac{d}{dr} + \frac{\ell (\ell +n - 2)}{r^2}\right] f(r)\right)Y_\ell(\theta).
\end{equation}
Recalling that the volume element in $\bbR^n$ in spherical coordinates is given by 
\begin{equation}
d^nx = r^{n-1} dr\, d^{n-1}\omega(\theta),
\end{equation}
it then follows that the differential operator   $\left((-\Delta)^2 +c|x|^{-4}\right)\big|_{C_0^{\infty}(\bbR^n \backslash \{0\})}$  is essentially self-adjoint
in $L^2(\bbR^n; d^nx)$ if and only if
\begin{equation}
 \left(\left[- \frac{d^2}{dr^2} - \frac{n-1}{r}\frac{d}{dr} + \frac{\ell (\ell + n - 2)}{r^2}\right]^2 +\frac{c}{r^4}\right)\Bigg|_{C_0^{\infty}((0,\infty))}
\end{equation}
is essentially self-adjoint in $L^2\big((0,\infty); r^{n-1}dr\big)$  for all $\ell\in \bbN_0$.
Next, consider the unitary map  
\begin{equation}
U: \begin{cases} L^2\big((0,\infty); r^{n-1}dr\big) \to  L^2((0,\infty);  dr), \\
f \mapsto r^{(n-1)/2} f. 
\end{cases} 
\end{equation}
A straightforward calculation shows that 
\begin{equation}
U \left(\left[- \frac{d^2}{dr^2} - \frac{n-1}{r}\frac{d}{dr} + \frac{\ell (\ell+ n - 2)}{r^2}\right]^2 +\frac{c}{r^4}\right)U^{-1} = \tau_{2,n,\ell}(c),
\end{equation}
where  $\tau_{2,n,\ell}(c)$ is the differential expression given by \eqref{E:tau_nl}.
Thus, it follows that  $\big((-\Delta)^2 +c|x|^{-4}\big)\big|_{C_0^{\infty}(\bbR^n \backslash \{0\})}$  is essentially self-adjoint in $L^2(\bbR^n; d^n x)$ if and only if $\tau_{2,n,\ell}(c)\big|_{C_0^{\infty}((0,\infty))}$   is essentially self-adjoint in $L^2((0,\infty);  dr)$
for all $\ell\in \bbN_0$.
\end{proof}
%%%%%%%%

%%%%%%%%
\begin{theorem} \lb{t3.2}
Let  $n\in \bbN$, $n\geq 2$. Then 
\begin{align}
& \left((-\Delta)^2 +c|x|^{-4}\right)\big|_{C_0^{\infty}(\bbR^n \backslash \{0\})}\ \mbox{is essentially self-adjoint  in $L^2(\bbR^n; d^n x)$}     \no \\
& \quad \text{if and only if}     \lb{E:condition} \\ 
& \;\; c\geq \begin{cases}
3(n+2)(6-n) & \mbox{for $2\leq n\leq 5$},   \\[10pt]
{\displaystyle -\frac{(n+4)n(n-4)(n-8)}{16}}& \mbox{for $n\geq 6$}.    
\end{cases}     \no
\end{align}
In particular, in the special case $c=0$, 
\begin{align}
\begin{split} 
& \text{$(- \Delta)^2\big|_{C_0^{\infty}(\bbR^n \backslash \{0\})}$ is essentially self-adjoint in $L^2(\bbR^n; d^n x)$}  \\
& \quad \text{if and only if $n \geq 8$.}
\end{split}
\end{align}  
\end{theorem}
%%%%%%%%
\begin{proof} Let  $n\in \bbN$, $n\geq 2$. 
By Lemma \ref{l3.1}, it suffices to determine for which $c\in \bbR$ the differential operator
$\tau_{2,n,\ell}(c)|_{C_0^{\infty}((0,\infty))}$ is essentially self-adjoint in $L^2((0,\infty);dr)$ for all $\ell\in \bbN_0$. 
A straightforward calculation shows that 
\begin{equation}
\tau_{2,n,\ell}(c)=\tau_2(c_1,c_2), 
\end{equation}
with
\begin{equation}\lb{E:c1c2}
c_1= -\frac{(n + 2\ell - 1)(n+2\ell-3)}{4}\quad \mbox{and}\quad    c_2=c_1^2+c.
\end{equation}
By Theorem~\ref{t2.2}, one then finds that 
\begin{align} 
\begin{split}
& \text{\it $\tau_{2,n,\ell}(c)|_{C_0^{\infty}((0,\infty))}$ is essentially self-adjoint in $L^2((0,\infty);dr)$} \\
& \quad \text{if and only if $c\geq \gamma_{2,n,\ell}$},  
\end{split}
\end{align}
%$\tau_{2,n,\ell}(c)\big|_{C_0^{\infty}((0,\infty))}$ is essentially self-adjoint if and only if $c\geq \gamma_{2,n,\ell}$,
where
%\begin{equation}\lb{2.9}
%c\geq  
%\gamma_{2,n,\ell}=
%\begin{cases}
%\hspace{3pc}45 + 12 c_1  & \mbox{for $c_1 \geq  -11/4$},\\[10pt]
%{\d -\frac{(3 + 4 c_1) (35 + 4 c_1)}{16}& \mbox{for $c_1 <  -11/4$},
%\end{cases}
%\end{equation}
%where $c_1=   -(n + 2\ell - 1)(n+2\ell-3)/4$. 
%
%More explictly, the constant $\gamma_{2,n,\ell}$ can be written as
\begin{equation}
\gamma_{2,n,\ell}=
\begin{cases}
-3 (n+2\ell+2)(n+2\ell -6) \quad \mbox{for}\  (n+2\ell-1)(n+2\ell-3)\leq 11,\\[5pt]
\d -\frac{(n+2\ell+4)(n+2\ell)(n+2\ell-4)(n+2\ell-8)}{16}  \\
\hspace*{2.3cm}  \mbox{for}\  (n+2\ell-1)(n+2\ell-3)\geq 11.
\end{cases}
\end{equation}

One notes that  when  $\ell=0$, one has $(n+2\ell-1)(n+2\ell-3)=(n-1)(n-3)\leq 11$ if and only if $n\leq 5$
and it follows that  $\gamma_{2,n,0}$ is equal to the right-hand side of the inequality in \eqref{E:condition}, that is, 
\begin{equation}
\gamma_{2,n,0} = \begin{cases}
3(n+2)(6-n) & \mbox{for $2\leq n\leq 5$},   \\[10pt]
{\displaystyle -\frac{(n+4)n(n-4)(n-8)}{16}}& \mbox{for $n\geq 6$}.    
\end{cases}
\end{equation} 
Next, we will prove that 
\begin{align}\lb{L:ell=0 implies}
\begin{split} 
& \text{\it if  $\tau_{2,n,\ell}(c)|_{C_0^{\infty}((0,\infty))}$ is essentially self-adjoint for $\ell=0$,} \\
&  \quad \text{\it  then it is  essentially self-adjoint for all $\ell\in \bbN_0$},    \\
\end{split} 
\end{align}
by showing that $\gamma_{2,n,\ell}\leq \gamma_{2,n,0}$ for all $\ell\in \bbN_0$. We recall that 
\begin{equation}
 \gamma_{2,n,\ell}= \gamma_{2,n+2\ell,0} \ \mbox{for all}\ \ell \in \bbN_0.
\end{equation} 
Thus, it suffices to show that the sequence $\{\gamma_{2,n,0}\}_{n \ge 2}$ is decreasing.
Table~\ref{T: gamma_2,n,0} lists the constants $\gamma_{2,n,0}$ for $2 \le n \le 12$. 
By inspection, we see that the finite sequence $\{\gamma_{2,n,0}\}_{2 \le n \le 8}$ is decreasing. 
For $n \ge 8$, we note that 
\begin{equation}
\gamma_{2,k+8,0} = -\frac{(k+12)(k+8)(k+4)k}{16}  \ \mbox{for}\ k \in \bbN_0,
\end{equation}
and it follows immediately that the sequence $\{\gamma_{2,n,0}\}_{n \ge 8}$ is decreasing. 
\end{proof}
%%%%%%%

%%%%%%% 
Table~\ref{T: gamma_2,n,0} shows the constants $\gamma_{2,n,0}$ for $2\leq n\leq 12$.
 
 \begin{table}[H]
 \begin{tabular}{c|ccccccccccc}
 $n$ & 2 &3 &4 &5&6& 7& 8& 9& 10& 11& 12\\[2pt]
\hline
\rule{0pt}{4ex}
$\gamma_{2,n,0}$ & 48&45&36&21&15&$\d\frac{231}{16}$&0&$\d-\frac{585}{16}$&$-105$&$\d-\frac{3465}{16}$&$-384$\\
\end{tabular}
\caption{The constants $\gamma_{2,n,0}$.}
\lb{T: gamma_2,n,0} 
\end{table}
%%%%%%%%

We conclude this section with a comparison of Theorem \ref{t3.2} and related results on (essential) self-adjointness results in the literature due to Okazawa, Tamura, and Yokota \cite{OTY11}, \cite{Ta10}. 

%%%%%%%%
\begin{remark} \lb{r3.3}
We start with lower semiboundedness properties before turning to (essential) self-adjointness results and hence recall the Hardy and two (higher--order) Rellich inequalities. (For an infinite sequence of such inequalities see Davies and Hinz \cite[Corollary~14]{DH98}.) Let $n \in \bbN$, then 
\begin{align}
\begin{split}
& \bigg[\f{(n-2)}{2}\bigg]^2 \int_{\bbR^n} d^nx \f{|f(x)|^2}{|x|^2} < \int_{\bbR^n} d^nx \, |((-\nabla f)(x)|^2,  \\
& \hspace*{4.5cm} f \in H^1(\bbR^n)\backslash\{0\}, \; n \geq 3, 
\end{split} \\
\begin{split}
& \bigg[\f{n(n-4)}{4}\bigg]^2 \int_{\bbR^n} d^nx \f{|f(x)|^2}{|x|^4} < \int_{\bbR^n} d^nx \, |((-\Delta f)(x)|^2,  \\
& \hspace*{4.75cm}f \in H^2(\bbR^n)\backslash\{0\}, \; n \geq 5,
\end{split}  \\
\begin{split}
& \bigg[\f{n(n-8)(n^2-16)}{16}\bigg]^2 \int_{\bbR^n} d^nx \f{|f(x)|^2}{|x|^8} 
< \int_{\bbR^n} d^nx \, \big|\big((-\Delta)^2 f\big)(x)\big|^2,  \\
& \hspace*{6.7cm} f \in H^4(\bbR^n)\backslash\{0\}, \; n \geq 9, 
\end{split}
\end{align}
implying 
\begin{align}
& H^1(\bbR^n) = \dom((|\nabla|) = \dom\big((-\Delta)^{1/2}\big) \subset \dom\big(|\dott|^{-1}\big), \quad n \geq 3,    \\
& H^2(\bbR^n) = \dom((-\Delta)) \subset \dom\big(|\dott|^{-2}\big), \quad n \geq 5,    \\
& H^4(\bbR^n) = \dom\big((-\Delta)^2\big) \subset \dom\big(|\dott|^{-4}\big), \quad n \geq 9.
\end{align}
In addition we recall the following lower boundedness properties,
\begin{align}
& \bigg((-\Delta) + \f{c}{|x|^2}\bigg)\bigg|_{C_0^{\infty}(\bbR^n\backslash\{0\})} \geq 0 \, \text{ if and only if } \,
c \geq - \bigg[\f{n-2}{2}\bigg]^2, \; n \geq 2,    \\
& \bigg((-\Delta)^2 + \f{c}{|x|^4}\bigg)\bigg|_{C_0^{\infty}(\bbR^n\backslash\{0\})} \geq 0 \, \text{ if and only if } \,
c \geq - \bigg[\f{n(n-4}{4}\bigg]^2, \; n \geq 3,    
\end{align}
and similarly, 
\begin{align}
& \bigg((-\Delta) + \f{c}{|x|^2}\bigg)\bigg|_{C_0^{\infty}(\bbR^n) \,\text{or}\, H^2(\bbR^n)} \geq 0 \, \text{ if and only if } \,
c \geq - \bigg[\f{n-2}{2}\bigg]^2, \; n \geq 5,    \\
& \bigg((-\Delta)^2 + \f{c}{|x|^4}\bigg)\bigg|_{C_0^{\infty}(\bbR^n) \,\text{or}\, H^4(\bbR^n)} \geq 0 \, \text{ if and only if } \, c \geq - \bigg[\f{n(n-4}{4}\bigg]^2, \; n \geq 9.   
\end{align}
For details on these facts, see, for instance, \cite[p.~213, 222]{BEL15}, \cite{Ok96}, \cite{OTY11}, and \cite[p.~90--101]{Re69}.

Turning to (essential) self-adjointness properties we now recall the following facts:
\begin{align}
\begin{split} 
& \bigg((-\Delta) + \f{c}{|x|^2}\bigg)\bigg|_{C_0^{\infty}(\bbR^n\backslash\{0\})} \, \text{ is essentially self-adjoint}  \lb{3.35} \\
& \quad \text{ if and only if } \, c \geq - \f{n(n-4)}{4}, \; n \geq 2,   \\
\end{split} \\
\begin{split}
& \bigg((-\Delta)^2 + \f{c}{|x|^4}\bigg)\bigg|_{C_0^{\infty}(\bbR^n\backslash\{0\})} \, \text{ is essentially self-adjoint} \\
& \quad \text{ if and only if } \, c \geq \begin{cases} 48 - 3 (n-2)^2, & 2\leq n \leq 5, \\
- [n(n-8)(n^2-16)]/16, & n \geq 6,    
\end{cases}    \lb{3.36} 
\end{split}
\end{align}
see \cite[Theorem~7.4.2]{EE18}, \cite{KSWW75}, \cite{KW72}, \cite{Sc72}, \cite{Si73} and Theorem \ref{t3.2}. Similarly,
\begin{align}
\begin{split} 
& \bigg((-\Delta) + \f{c}{|x|^2}\bigg)\bigg|_{C_0^{\infty}(\bbR^n)} \, \text{ is essentially self-adjoint}    \lb{3.37} \\
& \quad \text{ if and only if } \, c \geq - \f{n(n-4)}{4}, \; n \geq 5,    \\
\end{split} \\
\begin{split}
& \bigg((-\Delta)^2 + \f{c}{|x|^4}\bigg)\bigg|_{C_0^{\infty}(\bbR^n)} \, \text{ is essentially self-adjoint}    \lb{3.38} \\
& \quad \text{ if and only if } \, c \geq - [n(n-8)(n^2-16)]/16, \; n \geq 9,    
\end{split}
\end{align}
see  \cite[Proposition~7.4.1]{EE18}, \cite{KSWW75}, \cite{OTY11}, and \cite{Ta10}. In addition,
\begin{align}
\begin{split} 
& \bigg((-\Delta) + \f{c}{|x|^2}\bigg)\bigg|_{H^2(\bbR^n)\cap\dom(|\dott|^{-2})} \, \text{ is self-adjoint}    \lb{3.39} \\
& \quad \text{ if and only if } \, c > - \f{n(n-4)}{4}, \; n \geq 2,    \\
\end{split} \\
\begin{split} 
& \bigg((-\Delta) + \f{c}{|x|^2}\bigg)\bigg|_{H^2(\bbR^n)} \, \text{ is self-adjoint}     \lb{3.40} \\
& \quad \text{ if and only if } \, c > - \f{n(n-4)}{4}, \; n \geq 5,    \\
\end{split} \\
& \bigg((-\Delta) - \f{n(n-4)}{4|x|^2}\bigg)\bigg|_{H^2(\bbR^n)\cap\dom(|\dott|^{-2})}, \; n \geq  2, \, \text{ is essentially self-adjoint,}    \lb{3.41} \\
& \bigg((-\Delta) - \f{n(n-4)}{4|x|^2}\bigg)\bigg|_{H^2(\bbR^n)}, \; n \geq  5, \, \text{ is essentially self-adjoint,}    \lb{3.42}
\end{align}
and 
\begin{align}
\begin{split}
& \bigg((-\Delta)^2 + \f{c}{|x|^4}\bigg)\bigg|_{H^4(\bbR^n)\cap\dom(|\dott|^{-4})} \, \text{ is self-adjoint}    \lb{3.43} \\
& \quad \text{ if } \, c > \begin{cases} 112 - 3(n-2)^2, & 2\leq n \leq 8, \\
- [n(n-8)(n^2-16)]/16, & n \geq 9,    
\end{cases} 
\end{split} \\
\begin{split}
& \bigg((-\Delta)^2 + \f{c}{|x|^4}\bigg)\bigg|_{H^4(\bbR^n)} \, \text{ is self-adjoint}    \lb{3.44} \\
& \quad \text{ if } \, c > - [n(n-8)(n^2-16)]/16, \;  n \geq 9,    
\end{split} \\
\begin{split}
& \bigg((-\Delta)^2 + \f{112 - 3(n-2)^2}{|x|^4}\bigg)\bigg|_{H^4(\bbR^n)\cap\dom(|\dott|^{-4})}, \; 2 \leq n \leq 8,   \lb{3.45} \\ 
& \quad \text{ is essentially self-adjoint,} 
\end{split} \\
\begin{split}
& \bigg((-\Delta)^2 - \f{n(n-8)(n^2-16)}{16|x|^4}\bigg)\bigg|_{H^4(\bbR^n)}, \; n \geq 9,    \lb{3.46} \\ 
& \quad \text{ is essentially self-adjoint.} 
\end{split} 
\end{align}
Again we refer to \cite{OTY11} and \cite{Ta10}. \\[1mm] 
\noindent 
One notes that \eqref{3.36} (i.e., Theorem \ref{t3.2}) yields the constant $48$ instead of the larger constant $112$ in \eqref{3.45} (see also \eqref{3.43}) derived in \cite{OTY11} and \cite{Ta10}. \hfill $\diamond$
\end{remark}
%%%%%%%%

%%%%%%%%
%%%%%%%%
\section{On Strongly Singular Polyharmonic Homogeneous\\* Differential Operators} \lb{s4}
%%%%%%%%
%%%%%%%%

For  any $m\in \bbN$, it is natural to consider essential self-adjointness on the space $C_0^{\infty}(\bbR^n \backslash \{0\})$
of differential operators associated with strongly singular, homogeneous, polyharmonic differential expressions  of the type 
\begin{equation}
(-\Delta)^m +c|x|^{-2m}, \quad x\in \bbR^n \backslash \{0\},\quad n\in\bbN, \; n\geq 2,\quad m \in \bbN, \quad  c\in \bbR, 
\end{equation}
in $L^2(\bbR^n;d^nx)$. 
By the same separation of variables argument given in Section~\ref{s3}, one obtains the following result: 
%%%%%%%%
\begin{proposition} \lb{p4.1} The operator 
\begin{align}
& \left((-\Delta)^m +c|x|^{-2m}\right)\big|_{C_0^{\infty}(\bbR^n \backslash \{0\})}\ \mbox{\it is essentially self-adjoint  in $L^2(\bbR^n; d^n x)$}     \no \\
& \quad \text{\it if and only if}     \\
& \, \tau_{m,n,\ell}(c)|_{C_0^{\infty}((0,\infty))} \   \mbox{\it is essentially self-adjoint in $L^2((0,\infty); dr)$ for all $\ell\in \bbN_0$.}    \no
\end{align}
%where $\tau_{m,n,\ell}(c)$ is the differential expression given by
Here, 
\begin{align}\lb{E:tau_mnl} 
\begin{split}
\tau_{m,n,\ell}(c)= \left[-\frac{d^2}{dr^2} + \frac{(n + 2\ell - 1)(n+2\ell-3)}{4r^2}\right]^m+\frac{c}{r^{2m}},&   \\
 m, n \in \bbN, \; n \geq2, \quad \ell \in \bbN_0, \quad c \in \bbR.& 
 \end{split} 
\end{align}
\end{proposition}
%%%%%%%%

By induction on $m$, one obtains  that for $r>0$, $c\in \bbR$, and $z \in \bbC$,   
\begin{equation}\lb{E:Dmnl}
\tau_{m,n,\ell}(c)r^z = D_{m,n,\ell}(c; z)r^{z-2m},
\end{equation}
where $D_{m,n,\ell}(c; \dott)$ is the  polynomial of degree $2m$ given by
\begin{equation}
D_{m,n,\ell}(c; z)= (-1)^m\prod_{j=1}^{m} \left(z-\frac{n + 2 \ell+4j-5}{2}\right)\left(z+ \frac{n+2\ell-4j+1}{2} \right)+c. 
\lb{4.5} 
\end{equation}
Then, as before (see again  \cite[Remark~2.1]{GHT23}),  
\begin{align}
\begin{split} 
& \text{\it  $\tau_{m,n,\ell}(c)|_{C_0^{\infty}((0,\infty))}$ is essentially self-adjoint if and only if } \\
&  \quad \text{\it exactly $m$ roots  of $D_{m,n,\ell}(c; \dott)$ have real part $\leq -1/2$}\\
& \quad \text{\it and the remaining $m$ roots have real part $>-1/2$.}    \lb{4.6}
\end{split}  
\end{align}
For $c\in \bbR$, let the roots of $D_{m,n,\ell}(c;\dott)$ be denoted $\alpha_{m,n,\ell; j}(c)$, $j=1,\ldots,2 m$.
By the continuous dependence of the roots of a polynomial on the coefficients (see 
\cite[Theorem~(1.4)]{Ma66}), we may choose our labelling such that  each function  $\alpha_{m,n,\ell; j}(\dott)$ is continuous  and
\begin{equation} \lb{E:ordering}
\Re [\alpha_{m,n,\ell; 1}(c)] \leq  \Re [\alpha_{m,n,\ell; 2}(c)] \leq \cdots  \leq \Re [\alpha_{m,n,\ell; 2m}(c)],    
\quad c\in \bbR.
\end{equation}
It is easy to see from \eqref{4.5} that $D_{m,n,\ell}(c;\dott)$ is symmetric about $m - (1/2)$ and if $c=0$,   then the roots of $D_{m,n,\ell}(0;\dott)$ are real and symmetric about $m-(1/2)$.
It  follows that for all $c\in \bbR$, 
\begin{equation}\lb{E:symmetry}
\frac{1}{2} \left( \Re [\alpha_{m,n,\ell; j}(c)]  +\Re [\alpha_{m,n,\ell; 2m-j+1}(c)] \right) =m-(1/2),\quad 1\leq j\leq m.
\end{equation}
Furthermore, using Rouch\'e's  theorem, one shows (see \cite[Lemma~4.3]{GHT23}) 
\begin{equation}\lb{E:c to neg infinity}
 \lim_{c\to - \infty} \Re [\alpha_{m,n,\ell;m}(c)] = m-(1/2)
\end{equation}
and
\begin{equation}\lb{E:c to infinity}
 \lim_{c\to + \infty} \Re [\alpha_{m,n,\ell;m}(c)] = -\infty.
\end{equation}
In particular, by continuity there exist some $c\in \bbR$ such that $\Re [\alpha_{m,n,\ell;m}(c)]=-1/2$,
and it becomes natural to define
\begin{equation}
\gamma_{m,n,\ell}= \max\{ c\in \bbR \mid \Re [\alpha_{m,n,\ell;m}(c)]=-1/2\} <\infty.
\end{equation}
We note that if $c\geq  \gamma_{m,n,\ell}$, then $\Re [\alpha_{m,n,\ell;m}(c)]\leq -1/2$ and 
$\Re [\alpha_{m,n,\ell;m+1}(c)] >-1/2$ by \eqref{E:symmetry}.
Thus, by  \eqref{4.6} and \eqref{E:ordering},
\begin{equation}\lb{E:nec}
\mbox{\it if $c\geq  \gamma_{m,n,\ell}$, then $\tau_{m,n,\ell}(c)|_{C_0^{\infty}((0,\infty))}$ is essentially self-adjoint.}
\end{equation}
To say more about the constants $\gamma_{m,n,\ell}$, we can again employ the Routh--Hurwitz theory (see \cite[Ch.~XV]{Ga59}). Let $\widetilde{D}_{m,n,\ell}(c;\dott)$ be the polynomial given by 
\begin{equation}
\widetilde{D}_{m,n,\ell}(c; z)= D_{m,n,\ell}(c; z-1/2), \quad z\in \bbC.      \lb{4.13} 
\end{equation} 

%%%%%%
\begin{lemma} \lb{l4.2}
Let $c\in \bbR$. If the polynomial $D_{m,n,\ell}(c;\dott)$ has a root with real part equal to $-1/2$, then 
\begin{equation}\lb{E:Hurwitz}
\det (H_{m,n,\ell}(c))=0,
\end{equation}
where $H_{m,n,\ell}(c)$ is the $2m\times 2m$ Hurwitz matrix associated to $\widetilde{D}_{m,n,\ell}(c;\dott)$ in \eqref{4.13}.
In particular, $\gamma_{m,n,\ell}$ must be a root of  $\det (H_{m,n,\ell}(\dott))$.
% and hence  is  less or equal to the largest real root of $\det (H_{m,n,\ell}(\dott))$.
\end{lemma}
%%%%%% 
\begin{proof}
If $D_{m,n,\ell}(c;\dott)$ has a root with real part equal to $-1/2$, then $\widetilde{D}_{m,n,\ell}(c;\dott)$  has a root 
$\alpha\in \bbC$  with $\Re(\alpha)=(\alpha+\overline{\alpha})/2=0$. One notes that since $\widetilde{D}_{m,n,\ell}(c;\dott)$ has real coefficents,
$\overline{\alpha}$ is also a root of $\widetilde{D}_{m,n,\ell}(c;\dott)$. It now follows from Orlando'ts formula (see \cite[\S~XV.7]{Ga59}) that $\det (H_{m,n,\ell}(c))=0$. 
 \end{proof}
%%%%%%

We note that $\alpha=-D_{m,n,\ell}(0;-1/2)$ is a (real) root of $\det (H_{m,n,\ell}(\dott))$. Therefore, the polynomial $\det (H_{m,n,\ell}(\dott))$ can be factored as the product  
\begin{equation}\lb{E:factorization}
\det (H_{m,n,\ell}(z)) = \left(z + D_{m,n,\ell}(0; -1/2)\right) \cdot Q_{m,n,\ell}(z), \ z\in \bbC,
\end{equation}
where $Q_{m,n,\ell}(\dott)$ is a polynomial of degree $m-1$ with rational coefficients.

%%%%%% 
\begin{theorem}\lb{t4.2}
 Let  $n\in \bbN$, $n\geq 2$. Then 
\begin{align}
&\left((-\Delta)^3 +c|x|^{-6}\right)\big|_{C_0^{\infty}(\bbR^n \backslash \{0\})}\ \mbox{is essentially self-adjoint  in $L^2(\bbR^n; d^n x)$}     \no \\
& \quad \text{if and only if}     \lb{E:m=3} \\ 
& \;\; c\geq \begin{cases}
\displaystyle \frac{64}{27} \left(7112+504n -126 n^2 +(236+12n-3 n^2)\sqrt{964+60n-15 n^2}\right)  \\[5pt]
\hspace*{8.95cm} \mbox{for $2\leq n\leq 9$},  \\[5pt]
\displaystyle-\frac{(n+8)(n+4)n(n-4)(n-8)(n-12)}{64} \hspace*{2.15cm} \mbox{for $n\geq 10$}.\\
\end{cases}  \no
\end{align}
In particular, in the special case $c=0$, 
\begin{align}
\begin{split} 
& \text{$(- \Delta)^3\big|_{C_0^{\infty}(\bbR^n \backslash \{0\})}$ is essentially self-adjoint in $L^2(\bbR^n; d^n x)$}  \\
& \quad \text{if and only if $n \geq 12$.}
\end{split}
\end{align}  

\end{theorem}
 %%%%%%
\begin{proof} 
Let  $n\in \bbN$, $n\geq 2$. 
By \eqref{E:factorization}, the roots of $\det (H_{3,n,0}(\dott))$ are 
\begin{equation}\label{E:-D3n0} 
-D_{3,n,0}(0; -1/2)= - \frac{(n+8)(n+4)n(n-4)(n-8)(n-12)}{64} 
\end{equation}
and the roots of the quadratic polynomial $Q_{3,n,0}(\dott)$, which turn out to be
\begin{equation}\label{E:roots of Q3n0} 
 \frac{64}{27} \left(7112+504n -126 n^2 \pm(236+12n-3 n^2)\sqrt{964+60n-15 n^2}\right).
\end{equation}
We will see that $\gamma_{3,n,0}$ is equal to the largest real root of $\det(H_{3,n,0}(\dott))$ for all $n \geq 2$.
We note that if $n \geq 11$, then $964 + 60n - 15n^2 < 0$, and hence $Q_{3,n,0}(\dott)$ has no real roots. Thus, $n \geq 11$,  then $\det(H_{3,n,0}(\dott))$ has exactly one real root, namely $-D_{3,n,0}(0; -1/2)$.
By Lemma~\ref{l4.2}, $\gamma_{3,n,0}$ must be a real root of $\det(H_{3,n,0}(\dott))$, and therefore
\begin{equation}\lb{E : leq 12}
\gamma_{3,n,0} = -D_{3,n,0}(0; -1/2)\quad \text{for } n \geq 11.
\end{equation}
Next, we determine $\gamma_{3,n,0}$ for $2 \leq n \leq 10$. We again use Lemma~\ref{l4.2}, but we caution the reader that its converse does not hold: if $c \in \bbR$ is a real root of $Q_{3,n,0}(\dott)$, Lemma~\ref{l4.2} does \emph{not} imply that $D_{3,n,0}(c;\dott)$ necessarily has a root whose real part equals $-1/2$. Nevertheless, for $2 \leq n \leq 9$ (the case $n=10$ is somewhat exceptional; see below), if $c \in \bbR$ is the larger of the two real roots of $Q_{3,n,0}(\dott)$, as given by \eqref{E:roots of Q3n0}, then $D_{3,n,0}(c;\dott)$ does indeed have a complex root with real part equal to $-1/2$. This can be verified directly; for example, when $n=7$, \eqref{E:roots of Q3n0} shows that the larger of the two real roots of $Q_{3,7,0}(\dott)$ is
\begin{equation}\label{E:c for n=7}
c = \frac{64}{27}(4466 + 173\sqrt{649}).
\end{equation}
One notes that $D_{3,7,0}(c;\dott)$ is a sextic. However, by the symmetry $D_{3,7,0}(c;\dott)$ about $5/2$, finding the roots of $D_{3,7,0}(c;\dott)$ reduces to finding the roots of a cubic.
In fact, for  $z,w \in \bbC$ such that $w=(z-(5/2))^2$, we have
\begin{equation}
D_{3,7,0}(c;z)=0 \quad \mbox{if and only if }\quad 
w^3-\frac{107}{4} w^2 -\frac{2131}{16} w -\frac{2025}{64}  = c.
\end{equation}
Solving this cubic for the value of $c$ given by \eqref{E:c for n=7}, we find that $D_{3,7,0}(c;\dott)$ has the following six roots (which, perhaps surprisingly, are expressible in terms of square roots only):
\[
\begin{split}
 -\frac{1}{2} \pm \frac{1}{2} i &\sqrt{\frac{1}{3} \left(73 + 4 \sqrt{649}\right)}, \quad 
  \frac{5}{2} \pm \frac{1}{2} \sqrt{\frac{1}{3} \left(251 + 8 \sqrt{649}\right)}, \\
&\qquad \frac{11}{2} \pm \frac{1}{2} i \sqrt{\frac{1}{3} \left(73 + 4 \sqrt{649}\right)}.
\end{split}
\]
In particular, $D_{3,7,0}(c;\dott)$ has a pair of complex conjugate roots with real part equal to $-1/2$.  
The same kind of argument shows that $D_{3,n,0}(c;\dott)$ has this property if $c$ is the larger of the two real roots of $Q_{3,n,0}(\dott)$, for $2 \leq n \leq 9$.
Now, by  the definition of  $\gamma_{3,n,0}$ and by Lemma~~\ref{l4.2},  it follows that $\gamma_{3,n,0}$ is equal to the maximum of $D_{3,n,0}(0;-1/2)$  
and the largest real  root of $Q_{3,n,0}(\cdot)$, for $2\leq n \leq 9$. 

\begin{table}[h]
\begin{tabular}{r |r | r}
$n$ & $D_{3,n,0}(0;-1/2)$  & largest real  root of $Q_{3,n,0}(\cdot)$ 
\\[2pt]
\hline
\rule{0pt}{2.5ex}
2 & $225$   & {\bf 36864} \\
3 & $10395/64 \approx 162$ & $ {\bf64(7490 + 245\sqrt{1009})/27 \approx 36201}$ \\
4& $0$ & $ {\bf 64(7112 + 472\sqrt{241})/27 \approx 34227}$ \\
5 & $-12285/64 \approx -192$ & $ {\bf64(6482 + 221\sqrt{889})/27 \approx 30984}$ \\
6 & $-315 $ & ${\bf716800/27 \approx  26548}$ \\
7 & $-17325/64 \approx -271$ & ${\bf 64(4466 + 173\sqrt{649})/27 \approx  21033}$ \\
8 & $0$ &  ${\bf 394240/27\approx 14601}$ \\
9 & $29835/64 \approx 466$ & ${\bf 7488}$ \\
10 & ${\bf 945} $ & $0$ \\
11 & ${\bf 1028}$  & no real root\\
12 & ${\bf 0}$& no real root\\
13 & ${\bf -208845/64 \approx  -3263}$& no real root\\
14 & $ {\bf -10395}$& no real root\\
\end{tabular}
\caption{The constants $\gamma_{3,n,0}$,  $2 \leq n \leq 14$, are shown in bold. The approximations are obtained by rounding the exact entries to the nearest integer for easier comparison.}  
\lb{T:gamma3n0} 
\end{table}

For $n = 10$, neither of the two real roots $c$ of $Q_{3,10,0}(\dott)$ yields a root of  $D_{3,10,0}(c;\dott)$ whose real part equals $-1/2$, and hence, and hence $\gamma_{3,10,0} = -D_{3,10,0}(0;-1/2)$.

In summary, we have:
\begin{equation}
\gamma_{3,n,0} =
\begin{cases} 
\mbox{largest real root of $Q_{3,n,0}(\dott)$} & \mbox{for $2\leq n \leq 9$,}    \lb{4.22} \\
-D_{3,n,0}(0; -1/2)& \mbox{for $n \geq 10$.}
\end{cases}
\end{equation}
We will now show that  
\begin{equation}
\gamma_{3,n,\ell} < \gamma_{3,n,0} \quad \text{for all } n \ge 2, \ \ell \in \bbN.
\end{equation}
Since $\gamma_{3,n,\ell} = \gamma_{3,n+2\ell}$, it suffices to show that the sequence $\{\gamma_{3,n,0}\}_{n \ge 2}$ is decreasing.  
By inspection (see Table~\ref{T:gamma3n0}), we see that the finite sequence $\{\gamma_{3,n,0}\}_{2 \le n \le 12}$ is decreasing.  
For $n \ge 12$, we note that  
\begin{equation}\label{E : leq 12b}
\gamma_{3,k+12,0} = -\frac{(k+20)(k+16)(k+12)(k+8)(k+4)k}{64} \quad \text{for } k \in \bbN_0,
\end{equation}
and it follows immediately that the sequence $\{\gamma_{3,n,0}\}_{n \ge 12}$ is decreasing. 
\end{proof}

\begin{figure}[h]
\includegraphics[width=0.45\textwidth]{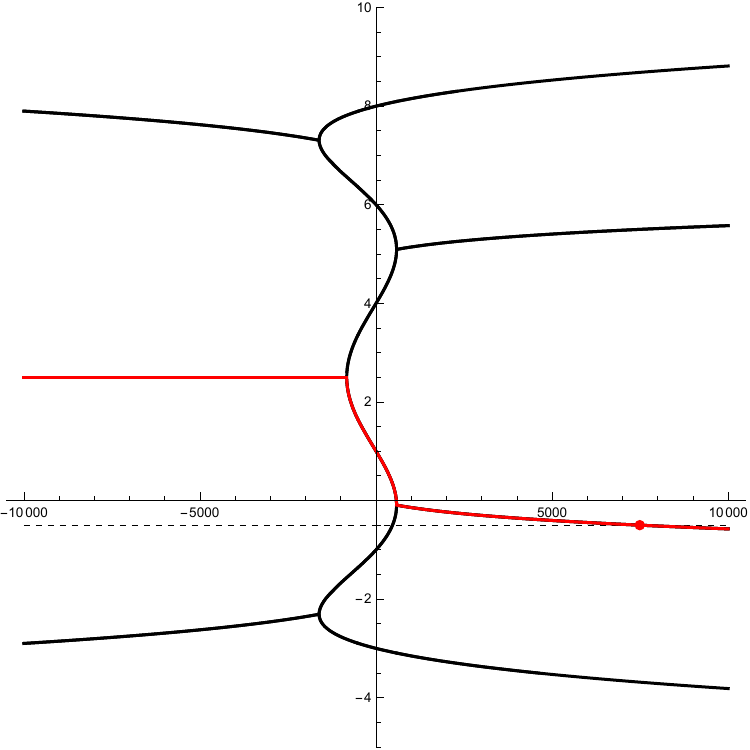}\qquad
\includegraphics[width=0.45\textwidth]{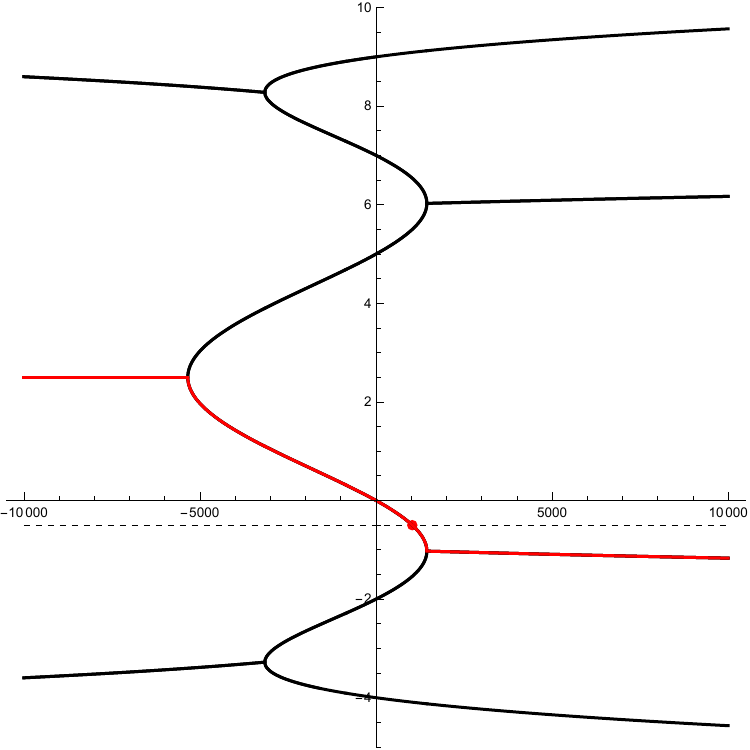}\\[1pc]
\caption{Graphs of  $\Re(\alpha_{3,n,0;j}(\dott))$, $1\leq j\leq 6$, for $n=9$ (left) and $n=11$ (right). The graph of  $\Re(\alpha_{3,n,0;3}(\dott))$, is shown in red for 
$n=9$ (left) and $n=11$ (right). 
The dashed horizontal line in both pictures is the line $y=-1/2$.  The graph of $\Re(\alpha_{3,n,0;3}(\dott))$ crosses the dashed line at the point $(\gamma_{3,n,0}, -1/2)=(7488,-1/2)$
for $n=9$ (left) and $(\gamma_{3,n,0}, -1/2)=(945,-1/2)$ for $n=11$ (right). }
\lb{F:m=3}
\end{figure}

%Table~\ref{} shows  
%
%
%\begin{equation}
%\frac{1}{2} \left(-1\pm i \sqrt{(73+4 \sqrt{649})/3}\right)
%\end{equation}
% 
%
%\begin{equation}
%w^3-\frac{107}{4} w^2 -\frac{2131}{16} w -\frac{2025}{64} \right) = c
% \end{equation}

%%%%%%

%\begin{equation}
%c\geq 
%\begin{cases}
%36864 & \mbox{for $n=2$}\\[10pt]
%\displaystyle \frac{64}{27} \left(7490+245 \sqrt{1009}\right) &\mbox{for $n=3$}\\[10pt]
%\displaystyle \frac{64}{27} \left(7112+472 \sqrt{241}\right) & \mbox{for $n=4$}\\[10pt]
%\displaystyle \frac{64}{27} \left(6482+221 \sqrt{889}\right) & \mbox{for $n=5$}\\[10pt]
%\displaystyle \frac{716800}{27} & \mbox{for $n=6$}\\[10pt]
%\displaystyle \frac{64}{27} \left(4466+173 \sqrt{649}\right) & \mbox{for $n=7$}\\[10pt]
%\displaystyle \frac{394240}{27} & \mbox{for $n=8$}\\[10pt]
%7488& \mbox{for $n=9$}\\[10pt]
%945 & \mbox{for $n=10$}\\[10pt]
%\displaystyle\frac{65835}{64} & \mbox{for $n=11$}\\[10pt]
%0 & \mbox{for $n=12$}\\[10pt]
%\displaystyle-\frac{(n+8)(n+4)n(n-4)(n-8)(n-12)}{64}&\mbox{for $n\geq 13$}\
%\end{cases}
%\end{equation}
\bigskip
This leads to the following natural question: For general $m,n\in \bbN$, $n\geq 2$, 
\begin{align}
\begin{split} 
& \text{\it does there exist $c_{m,n} \in \bbR$ such that } 
 \big((-\Delta)^m +c|x|^{-2m}\big)\big|_{C_0^{\infty}(\bbR^n \backslash \{0\})},  \; c \in \bbR, \\
& \quad \text{\it is essentially self-adjoint in $L^2(\bbR^n;d^nx)$ if and only if $c \geq c_{m,n}$?}
\end{split} \lb{guess}
\end{align}
%In particular, if $(-\Delta)^m\big|_{C_0^{\infty}(\bbR^n \backslash \{0\})}$ is essentially self-adjoint, does it follow that 
%$\left((-\Delta)^m +c|x|^{-2m}\right)\big|_{C_0^{\infty}(\bbR^n \backslash \{0\})}$  is essentially self-adjoint for all $c\geq 0$
More specifically, 
\begin{equation}
\text{\it Does \, $c_{m,n} = \gamma_{m,n,0}$ \, hold\,?}    \lb{guess2}
\end{equation}

Perhaps surprisingly, the answer to questions \eqref{guess2} as well as \eqref{guess} is negative for some $m,n\in \bbN$, $n\geq 2$, $m\geq 5$, even if $\gamma_{m,n,\ell}<\gamma_{m,n,0}$
for all $\ell \in \bbN$. The reason for this is the fact that the function $\Re [\alpha_{m,n,0; m}(\dott)]:\bbR\to \bbR$ is not necessarily weakly decreasing. For example, $\Re [\alpha_{5,20,0; 5}(\dott)]$ is not weakly decreasing as shown in  Fig.~\ref{F: m=5, n=20}.

%%%%%%
\begin{figure}[H] 
\includegraphics[width=0.95\textwidth]{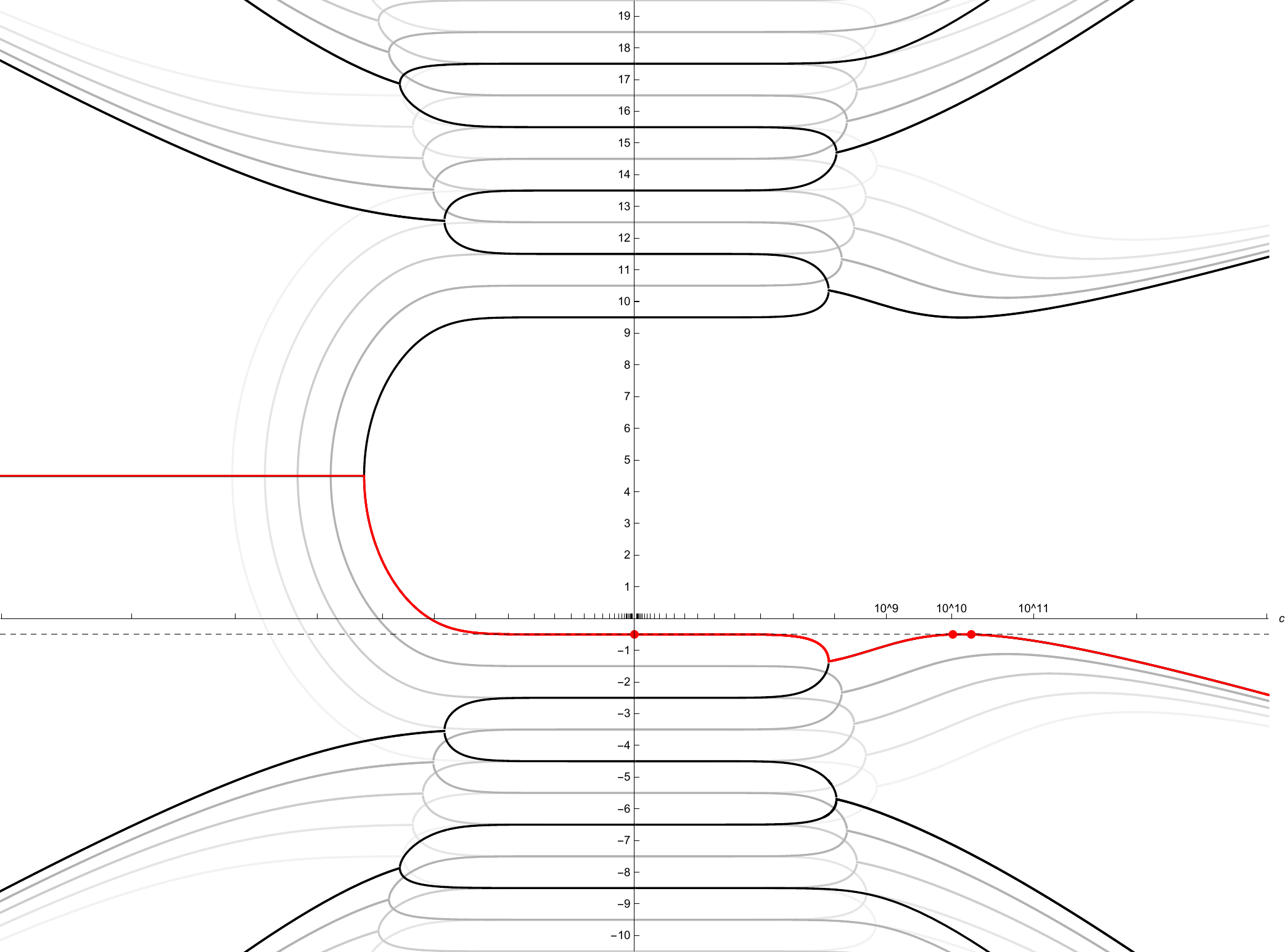}\qquad
\caption{Graphs of  $\Re(\alpha_{5,20,\ell;j}(\dott))$, $1\leq j\leq 10$, $0\leq \ell\leq 4$. The graph of  $\Re(\alpha_{5,20,0;5}(\dott))$ is shown in red. The graphs for $1\leq \ell \leq 4$ are shown in gray with the graphs for $\ell=4$ being the lightest.~The dashed horizontal line  is the line $y=-1/2$. We note that the graph $\Re(\alpha_{5,20,0;5}(\dott))$
crosses the dashed line at three points, shown as red dots. These points are $(\alpha,-1/2)$,  $(\beta,-1/2)$, and $(\gamma,-1/2)$, where $\alpha=0$, $\beta  \approx 1.0436\times 10^{10}$, and $\gamma=\gamma_{5,20,0} \approx 1.8324 \times 10^{10}$ are the three real roots of the quintic polynomial $\det (H_{5,20,0}(\dott))$. 
}
\lb{F: m=5, n=20}
\end{figure}
%%%%%%

%In the following we will consider, in detail, the case when $n=20$ and $m=5$. 

%%%%%%%
 \begin{theorem} \lb{t4.4}
Let  $m=5$, $n=20$, $c \in \bbR$. Then 
\begin{align}
\begin{split} 
& \left((-\Delta)^5 +c|x|^{-10}\right)\big|_{C_0^{\infty}(\bbR^{20} \backslash \{0\})}\ \mbox{is essentially self-adjoint  in $L^2\big(\bbR^{20}; d^{20}x\big)$}    \\ 
& \quad \text{if and only if $c \in [\alpha,\beta] \cup [\gamma,\infty)$}, 
\end{split} 
\end{align}
where $\alpha = 0$, $\beta \approx 1.0436 \times 10^{10}$, and $\gamma = \gamma_{5,20,0} \approx 1.8324 \times 10^{10}$ are the three real roots of the quintic polynomial $\det\big(H_{5,20,0}(\,\cdot\,)\big)$. More precisely, $\beta$ and $\gamma$ are the two real roots of the quartic $Q_{5,20,0}(\,\cdot\,)$, or equivalently, the roots of the quartic equation
\begin{align}
&3125 z^4-83914629120000 z^3+429438995162964368031744 z^2    \no \\
&\quad +1045471534388841527438982355353600 z       \lb{4.26} \\  
&\quad +629847004905001626921946285352115240960000 =0    \no 
\end{align}
$($see the paragraph preceding \eqref{E:Re =-1/2}$)$. 
Thus, this example now exhibits the $($nontrivial\,$)$ ``island'' of non-self-adjointness for $c \in (\beta,\gamma)$.
\end{theorem}
%%%%%%%
\begin{proof} The methods used in the proof are very similar to the ones in the proof of Theorem~\ref{t4.2}. However, since the polynomials $Q_{5,20,\ell}(\dott)$, $\ell \in \bbN_0$,
are  quartic polynomials, the sign of the discriminant alone no longer determines how many roots of  $Q_{5,20,\ell}(\dott)$ (and hence  of $\det (H_{5,20,0}(\dott))$) are real.

Consider a general quartic polynomial $az^4+bz^3+cz^2+dz+e$ in $z\in \bbC$ with coefficients $a,b,c,d,e\in \bbR$, $a\not=0$, and introduce the following three real quantities:
\begin{align}
& \operatorname{Disc}_z(az^4+bz^3+cz^2+dz+e)=256 a^3 e^3-192 a^2 b d e^2-128 a^2 c^2 e^2 \\ 
&\quad +144 a^2 c d^2 e-27 a^2 d^4+144 a b^2 c e^2-6 a b^2 d^2 e-80 a b c^2 d e+18 a b c d^3\no \\
&\quad +16 a c^4 e-4 a c^3 d^2-27 b^4 e^2+18 b^3 c d e-4 b^3 d^3-4 b^2 c^3 e+b^2 c^2 d^2,\no \\[5pt]
& \Pi_z(az^4+bz^3+cz^2+dz+e)=8ac-3b^2, \lb{E:Pi}\\[5pt] 
& \Lambda_z(az^4+bz^3+cz^2+dz+e)=64 a^3 e-16 a^2 b d-16 a^2 c^2+16 a b^2 c-3 b^4. \lb{E:Lambda}
\end{align}
Knowledge of the signs of these three quantities (see \cite[p. 45]{Di14}) completely determines how many of the roots of $az^4+bz^3+cz^2+dz+e=0$
are real.
For example (and these are the only relevant cases in what follows),
one has:
\begin{align}
&  \text{\it If $ \operatorname{Disc}<0$, then two roots are real and two are imaginary;}\lb{E:two real}\\[5pt]
& \text{\it  if $ \operatorname{Disc}>0$ and $\Pi\geq 0$ or $\Lambda\geq 0$, then there are no real roots.}\lb{E:no real}
\end{align}
By using \eqref{E:two real} and \eqref{E:no real}, we will show that  $Q_{5,20,\ell}(\dott)$ has two (distinct) real roots for $\ell=0$ and no real roots for $\ell\geq 1$.

First, with the help of a computer algebra system, we  find that 
\begin{align}  \lb{E:Discriminant}
&\operatorname{Disc}_z( Q_{5,20,\ell}(z))=\\ \no
&\quad 2^{206}\cdot 3^{24}\cdot  5^{37}\cdot 7^{6}\cdot  (\ell+3)^2 (\ell+4)^2 (\ell+5)^2 (\ell+13)^2 (\ell+14)^2 (\ell+15)^2 \\ \no
&\quad \cdot (42025 \ell^{12}+4538700 \ell^{11}+206928725 \ell^{10}+5143646250 \ell^9+74944977045 \ell^8\\ \no
&\quad\quad+637790320440 \ell^7+2964667430315 \ell^6+6583242809250 \ell^5+6872374126090 \ell^4\\ \no
&\quad\quad-163522446480 \ell^3-11647424700360 \ell^2-13287788700480 \ell-9475687380096)^2\\  \no
&\quad\cdot (-201001728116736 + 27683121149032320 \ell + 57207795364578740 \ell^2  \\ \no
&\quad\quad+ 48348657048032820 \ell^3 + 22496586334593565\ell^4  \\ \no
&\quad\quad+  6467555103057000 \ell^5 + 1216475264646540 l^6 + 153854808485040 \ell^7  \\ \no
&\quad\quad+  13170130070470 \ell^8 + 752206815000 \ell^9 + 27444944600 \ell^{10}  \\ \no
&\quad\quad+  578396700 \ell^{11} + 5355525 \ell^{12}) 
\end{align}
One notes that  $2^{206}\cdot 3^{24}\cdot  5^{37}\cdot 7^{6}\cdot  (\ell+3)^2 (\ell+4)^2\cdots  (\ell+15)^2$ is positive for all $\ell \in   \bbN_0$.
Similarly, the factor  $(42025 \ell^{12}+4538700 \ell^{11}+\cdots -9475687380096)^2$ of \eqref{E:Discriminant} is nonnegative for all $\ell \in \bbN_0$.
The factor is in fact nonzero and hence positive for all $\ell \in \bbN_0$. For $\ell=0$ and $\ell=1$, this can be checked by direct evaluation.
For $\ell\geq 2$, write $ \ell =k+2$ with  $k\in \bbN_0$. It can then be checked that $42025 \ell^{12}+4538700 \ell^{11}+\cdots -9475687380096$
is a polynomial expression in $k$ with positive integer coefficients and hence the factor in question is indeed nonzero for all $\ell \in \bbN_0$.
The last factor $(-201001728116736 + 27683121149032320 \ell +\cdots + 5355525 \ell^{12})$ of  \eqref{E:Discriminant} is a polynomial expression in $\ell$ 
with positive integer coefficients except for the constant term.
Since $27683121149032320>201001728116736$, it follows that this factor is negative for $\ell=0$ and positive for $\ell\geq 1$.
Thus we proved that  $\operatorname{Disc}_z( Q_{5,20,\ell}(z))$ is negative for  $\ell=0$ and positive for $\ell\geq 1$.

Let us consider what this means for $\ell=0$. Since $\operatorname{Disc}_z( Q_{5,20,0}(z))<0$ we find, by \eqref{E:two real}, that $Q_{5,20,0}(\dott)$ has two (distinct) real roots and hence,
by \eqref{E:factorization}, $\det (H_{5,20,0}(\dott))$ 
has three real roots namely  $\alpha=-D_{5,20,0}(0; -1/2)=0$ and the two real  roots $\beta  \approx 1.0436\times 10^{10}$, and $\gamma \approx 1.8324 \times 10^{10}$ of    $Q_{5,20,0}(\dott)$.
We claim that 
\begin{equation}\lb{E:Re =-1/2}
\mbox{\it $\Re(\alpha_{5,20,0;5}(c)) = -1/2$ if and only if $c\in \{0,\beta,\gamma\}$.}
\end{equation}
By Lemma~\ref{l4.2}, it suffices to prove the right-to-left implication. %if $c\in \{0,\beta,\gamma\}$, then $\Re(\alpha_{5,20,0;5}(c)) = -1/2$.
By \eqref{E:Dmnl},  
\begin{align}
\begin{split}
&\left(\alpha_{5,20,0;j}(0)\right)_{1\leq j\leq 10}  \\
&\qquad=(-17/2,-13/2,-9/2,-5/2,-1/2,19/2,23/2,27/2,31/2,35/2).
\end{split}
\end{align}
Thus, $\Re(\alpha_{5,20,0;5}(0)) = \alpha_{5,20,0;5}(0)=-1/2.$
It remains to show $\Re(\alpha_{5,20,0;5}(\beta)) = -1/2$ and $\Re(\alpha_{5,20,0;5}(\gamma)) = -1/2$.
By finding approximate roots of $D_{5,20,0}(1.5\times 10^{10}; \dott)$,   
\begin{align}
&\left(\Re\big(\alpha_{5,20,0;j}\big(1.5\times 10^{10}\big)\big)\right)_{1\leq j\leq 10}  \\
&\quad\approx(-10.03,-7.326,-7.326,-0.496,-0.496,9.496,9.496,16.33,16.33,19.03)\no 
\end{align}
In particular, $\Re\big(\alpha_{5,20,0;5}\big(1.5\times 10^{10}\big)\big)\approx -0.496 >-1/2$. In the same way one finds 
 $\Re\big(\alpha_{5,20,0;5}\big(0.5\times 10^{10}\big)\big)\approx -0.555 <-1/2$.
Therefore, by continuity, there must exist some $c\in (0.5\times 10^{10},1.5\times 10^{10})$ such that $\Re(\alpha_{5,20,0;5}(c)) = -1/2$. By Lemma~\ref{l4.2}, $\det (H_{5,20,0}(c))=0$ and hence $c=\beta$ since
$\beta$ is the only root of  $\det (H_{5,20,0}(\dott))$ in the interval $(0.5\times 10^{10},1.5\times 10^{10})$.
Similarly, one shows $\Re(\alpha_{5,20,0;5}(\gamma)) = -1/2$.

By \eqref{E:Re =-1/2},  since $\Re\big(\alpha_{5,20,0;5}\big(0.5\times 10^{10}\big)\big)< -1/2$ ,  $\Re\big(\alpha_{5,20,0;5}\big(1.5\times 10^{10}\big)\big) > -1/2$,  \eqref{E:c to neg infinity}, and \eqref{E:c to infinity}, 
it follows that $\Re(\alpha_{5,20,0;5}(c)) \leq -1/2$ if and only if $c \in [0,\beta] \cup [\gamma,\infty)$,
which, by \eqref{E:symmetry} and \eqref{4.6}, implies that 
\begin{align}\lb{E:tau5 ell=0}
\begin{split} 
& \mbox{\it $\tau_{5,20,0}(c)|_{C_0^{\infty}((0,\infty))}$ is essentially self-adjoint}\\
&  \quad \text{\it  if and only if $c \in [0,\beta] \cup [\gamma,\infty).$} \\
\end{split} 
\end{align}
 
 By \eqref{E:tau_mnl}, the theorem follows once we show that  $\tau_{5,20,\ell}(c)\big|_{C_0^{\infty}((0,\infty))}$ is essentially self-adjoint in $L^2((0,\infty);dr)$ for all $\ell\in \bbN_0$, $\ell\geq 1$, and  $c\in [0,\infty)$. 
To prove the latter, we begin by showing that $Q_{5,20,\ell}(\dott)$ has no real roots for all $\ell\geq 1$. We consider two cases. First, suppose  $\ell \geq 29$.
Then we can write $\ell=k+29$ for some $k\in \bbN_0$, and the quantity $\Pi_z( Q_{5,20,\ell}(z))$ given by \eqref{E:Pi} has the following expansion:
\begin{align}
&\Pi_z( Q_{5,20,\ell}(z))=2^{41}\cdot 3^4 \cdot 5^{15}\\ \no
&\quad \cdot (161875 k^{10}+61512500 k^9+10201465000 k^8+969468160000 k^7\\ \no
&\quad\quad+58206830051875 k^6+2291590504307500  k^5+59262332963402100 k^4\\ \no
&\quad\quad+974749919610039200 k^3+9364063767203524800 k^2\\ \no
&\quad\quad+42256876792510195200k+32928178597910728704).\no
\end{align}
This is a polynomial expression in $k$ with positive integer coefficients and hence $\Pi_z( Q_{5,20,\ell}(z))>0$ for all $\ell \geq 29$.
Thus, by \eqref{E:no real},  $Q_{5,20,\ell}(\dott)$  has no real roots for all $\ell \geq 29$.
For $1\leq \ell \leq 28$, it is possible that $\Pi_z( Q_{5,20,\ell}(z))<0$ and we also need to consider the sign of  $\Lambda_z( Q_{5,20,\ell}(z))$ as given by \eqref{E:Lambda}.
We recorded the  signs of the relevant quantities  in Table~\ref{T: signs}.  
\begin{table}[H]
 \begin{tabular}{c |cccccccccccccccc}
 $\ell$ & 0 & 1 &   2 &3 &4 &5&6& 7& 8& 9& 10& $\cdots$ & 27 & 28& 29& 30\\[2pt]
\hline
\\[-10 pt]
$\operatorname{Disc}$ & $-$ & $+$ & $+$ & $+$ & $+$ & $+$ & $+$ & $+$ &$+$ &$+$ &$+$ &$\cdots$&$+$ &$+$&$+$&$+$\\
$\Pi$ &$-$&$-$&$+$&$+$&$+$&$+$&$+$&$+$&$+$&$-$&$-$ &$\cdots$&$-$ &$-$&$+$&$+$ \\
$\Lambda$ &$-$&$+$&$+$&$-$&$-$&$-$&$-$&$+$&$+$&$+$&$+$&$\cdots$&$+$&$+$  &$+$&$+$
\end{tabular}
\caption{The signs of  the quantities $\operatorname{Disc}_z(Q_{5,20,\ell}(z))$, $\Pi_z(Q_{5,20,\ell}(z))$, and $\Lambda_z(Q_{5,20,\ell}(z))$, $0\leq \ell\leq 30$.
The omitted entries, indicated by ellipses, all have the same sign as their neighbors in the same row.}
\lb{T: signs} 
\end{table}
\noindent
By \eqref{E:no real},   $Q_{5,20,\ell}(\dott)$  also has no real roots for $1\leq \ell \leq 28$.
Thus, by \eqref{E:factorization},   $\det (H_{5,20,\ell}(\dott))$ has exactly one real root, namely $-D_{5,20,\ell}(0; -1/2)$, for all $\ell\geq 1$.
 Next, by an argument as in the proof of \eqref{E : leq 12b}, one can show that 
\begin{equation}
 D_{5,20,\ell}(0; -1/2)>D_{5,20,0}(0; -1/2)=0\ \mbox{for all $\ell\geq 1$.}
\end{equation}
In light of the above, this means that $\gamma_{5,20,\ell} < 0$ for all $\ell\geq 1$. By \eqref{E:nec}, we then find that
 $\tau_{5,20,\ell}(c)\big|_{C_0^{\infty}((0,\infty))}$ is essentially self-adjoint in $L^2((0,\infty);dr)$ for all $\ell\in \bbN_0$, $\ell\geq 1$, and  $c\in [0,\infty)$, which
together with \eqref{E:tau5 ell=0} implies the theorem.
\end{proof}
%%%%%%

For general $m,n\in \bbN$, $n\geq 2$, it appears difficult to determine in a systematic manner for which $c\in \bbR$ the operator $\left((-\Delta)^m +c|x|^{-2m}\right)\big|_{C_0^{\infty}(\bbR^n \backslash \{0\})}$ is essentially self-adjoint in $L^2(\bbR^n;d^nx)$. However, in the physically relevant case $n=3$ we offer the following:

%%%%%% 
\begin{conjecture} \lb{c4.5}
For $m\in \bbN$, $n=3$, and $c\in \bbR$,
\begin{align}
\begin{split}
& \left((-\Delta)^m +c|x|^{-2m}\right)\big|_{C_0^{\infty}(\bbR^3 \backslash \{0\})}\ \mbox{is essentially self-adjoint  in $L^2( \bbR^3; d^3x)$}  \\
& \quad \text{if and only if   $c\geq \gamma_{m,3,0}$.}
\end{split}
\end{align} 
Furthermore, asymptotically,
\begin{equation}\lb{E:asympt}
\gamma_{m,3,0} \underset{m\to\infty}{\sim}  \big(2m^2\big/\pi\big)^{2m}.
\end{equation}
\end{conjecture}
%%%%%%

%%%%%%
\begin{remark} \lb{r4.6} 
In \cite{GHT23},  we proved that  $\tau_{m,3,0}(c)\big|_{C_0^{\infty}((0,\infty))}$  is essentially self-adjoint  if and only if $c\geq \gamma_{m,3,0}$ for all $m\in \bbN$.
In \cite[Appendix]{GHT23}, we also gave 
a heuristic argument why one should expect \eqref{E:asympt} to be true.
The first part of the conjecture would be true if   $\gamma_{m,3,\ell}\leq \gamma_{m,3,0}$ for all $\ell \in \bbN_0$. The latter is more subtle than one might initially think since
it is, in general, \emph{not} true that $\Re [\alpha_{m,n,\ell;m}(c)] \leq \Re [\alpha_{m,n,0;m}(c)]$ for all $c\in \bbR$. Using Mathematica, we verified that $\gamma_{m,3,\ell}\leq \gamma_{m,3,0}$ for $1\leq \ell,m\leq 30$. 
\hfill $\diamond$ 
\end{remark}
%%%%%%

%%%%%%
\begin{remark} \lb{r4.7} 
In the special case $c=0$, one can show, that for $s \in (0,\infty)$, $n\in\bbN$,
\begin{align}
\begin{split}
& (-\Delta)^s\big|_{C_0^{\infty}(\bbR^n \backslash \{0\})}\ \mbox{is essentially self-adjoint in $L^2( \bbR^n; d^nx)$}  \\
& \quad \text{if and only if $n \geq 4s$.}    \lb{4.39} 
\end{split}
\end{align} 
This follows from Faris \cite[p.~33--35]{Fa75}, who uses the Fourier transform and elements of tempered distributions (he studies the quadratic form domain, i.e., $s=1/2$, and the operator domain, i.e., $s=1$, but his method extends to $s \in (0,\infty)$). In the case $s = m \in \bbN$, the fact \eqref{4.39} also follows from \eqref{4.5} and \eqref{4.6}. 

The case $s=1/2$ shows that 
\begin{align}
\begin{split} 
& \text{$C_0^{\infty}(\bbR^n \backslash \{0\})$ is a form core for $(-\Delta)|_{H^2(\bbR^n)}$ in $L^2( \bbR^n; d^nx)$}   \\
& \quad \text{if and only if $n \geq 2$.}
\end{split} 
\end{align} 

\hfill $\diamond$ 
\end{remark}
%%%%%%

%%%%%%%%
%%%%%%%%
\appendix
%%%%%%%%
%%%%%%%%
\section{A Fundamental System of Solutions of $\tau_2(c_1,c_2) y = \lambda y$} \lb{sA}
%%%%%%%%
%%%%%%%%

The goal of this appendix is to describe a fundamental system of solutions 
of the fourth order differential equation
\begin{equation} \lb{A.1}
\tau_2(c_1,c_2) y(\lambda;r) = 
\left[\frac{d^4}{dr^4} + c_1\left(\frac{1}{r^2}\frac{d^2}{dr^2} +\frac{d^2}{dr^2}\frac{1}{r^2}\right) +  \frac{c_2}{r^4}\right]y(\lambda;r) =\lambda\, y(\lambda;r)     
\end{equation}
for all $(c_1,c_2)\in \bbR^2$ and spectral parameter $\lambda \in \bbC$. 

We recall that the roots $\alpha_j(c_1,c_2)$ of the characteristic equation associated to the homogenous equation \eqref{A.1} (i.e., $\lambda=0$ in \eqref{A.1}) are of the form \eqref{2.12}--\eqref{2.15}. In the following, for $1\leq j\leq 4$, we will often just   write  $\alpha_j$ instead of $\alpha_j(c_1,c_2)$ to simplify notation.

One knows from the outset (see, e.g., \cite{GH23} and the references therein) that if $(\alpha_j-\alpha_{j'})/4\not\in -\bbN_0$ for all $1\leq j<j'\leq 4$, then \eqref{A.1} has the following fundamental system of solutions,
\begin{align}
y_1(\lambda;r)&=r^{\alpha_1} \,_{0} F_3\left(\!\begin{array}{c}\\
{ 1+\frac{\alpha_1-\alpha_2}{4}, 1+\frac{\alpha_1-\alpha_3}{4}, 1+\frac{\alpha_1-\alpha_4}{4}}
\end{array} \bigg\vert\,  \frac{\lambda r^4}{256} \right),\\
y_2(\lambda;r)&=r^{\alpha_2} \,_{0} F_3\left(\!\begin{array}{c}\\
{ 1+\frac{\alpha_2-\alpha_1}{4}, 1+\frac{\alpha_2-\alpha_3}{4}, 1+\frac{\alpha_2-\alpha_4}{4}}
\end{array} \bigg\vert\,  \frac{\lambda r^4}{256} \right),\\
y_3(\lambda;r)&=r^{\alpha_3} \,_{0} F_3\left(\!\begin{array}{c}\\
{ 1+\frac{\alpha_3-\alpha_1}{4}, 1+\frac{\alpha_3-\alpha_2}{4}, 1+\frac{\alpha_3-\alpha_4}{4}}
\end{array} \bigg\vert\,  \frac{\lambda r^4}{256} \right),\\
y_4(\lambda;r)&=r^{\alpha_4} \,_{0} F_3\left(\!\begin{array}{c}\\
{ 1+\frac{\alpha_4-\alpha_1}{4}, 1+\frac{\alpha_4-\alpha_2}{4}, 1+\frac{\alpha_4-\alpha_3}{4}}
\end{array} \bigg\vert\,  \frac{\lambda r^4}{256} \right),
\end{align}
where $\,_{0} F_3\left(\!\begin{array}{c}\\ { a, b, c} \end{array} \bigg\vert\,  z \right)$ denotes a generalized hypergeometric function (see, e.g., \cite[Ch.~IV]{EMOT53}, \cite[Ch.~16]{OLBC10}).

%%%%%%%
\begin{lemma} \lb{lA.1}
Let $(c_1,c_2)\in \bbR^2$ and let $k\in \bbN_0$.
\begin{itemize}
\item[(a)]
If $\d \frac{\alpha_2-\alpha_3}{4}=-k$ or $\d \frac{\alpha_1-\alpha_4}{4}=-k$, then
\begin{equation}\lb{E:equation of line}
 c_2=1 - 4 c_1 + c_1^2  + 16 c_1 k^2 - 20 k^2  + 64 k^4 .
\end{equation}
\item[(b)]
If $\d \frac{\alpha_1-\alpha_2}{4}=-k$,  then
\begin{equation}\lb{E:equation of parabola}
\quad c_2= \frac{-9 - 24 c_1 - 128 c_1 k^2 + 160 k^2  - 256 k^4}{16} .
\end{equation}
\end{itemize}
\end{lemma}
%%%%%%%
\begin{proof}
The proof is straightforward and hence omitted.
\end{proof}
%%%%%%%

For $k\in \bbN_0$, define a parabola $\mathbb{P}_k \subset \bbR^2$ and a line $\mathbb{L}_k \subset \bbR^2$ as follows:
\begin{align}
\mathbb{P}_k&=\{(c_1,c_2)\in \bbR^2 \mid \mbox{\eqref{E:equation of line} holds}\},\\[5pt]
\mathbb{L}_k&=\{(c_1,c_2)\in \bbR^2 \mid \mbox{\eqref{E:equation of parabola} holds}\}.
\end{align} 
These families of parabolas and lines satisfy following remarkable properties (see Fig.~\ref{F:parabolas and lines}), which are easy to verify:

\begin{itemize}
\item For $k\in \bbN_0$, the line  $\mathbb{L}_k$ is tangent to the parabola $\mathbb{P}_0$.\\[-5pt]
\item For $k\in \bbN_0$, the  parabola $\mathbb{P}_k$ is tangent to the  line  $\mathbb{L}_0$.\\[-5pt]
\item For $h,k\in \bbN_0$, $h\not= k$, if  $(c_1,c_2)\in \mathbb{L}_{h}\cap \mathbb{L}_{k}$, then
\begin{equation}
c_1=\frac{5}{4}- 2 h^2 - 2k^2.
\end{equation}
\item For $h,k\in \bbN_0$, $h\not= k$, if  $(c_1,c_2)\in \mathbb{P}_{h}\cap \mathbb{P}_{k}$, then 
\begin{equation}
c_1=\frac{5}{4}- 4 h^2 - 4 k^2.
\end{equation}
\item For $h,k\in \bbN_0$, $h\not= k$, if  $(c_1,c_2)\in \mathbb{L}_{h}\cap \mathbb{P}_{k}$, then 
\begin{equation}
c_1=\frac{5}{4} - 4 h^2\pm  8 h k - 8 k^2.
\end{equation}
\end{itemize}

The following lemma can be viewed as a converse of Lemma~\ref{lA.1}.

%%%%%%%
\begin{lemma} \lb{lA.2} 
Let $(c_1,c_2)\in \bbR^2$ and let $k\in \bbN_0$. 
\begin{itemize}
\item[(a)] If $(c_1,c_2)\in \mathbb{L}_k$, then 
\begin{align}
\frac{\alpha_1-\alpha_4}{4}=-k &\quad  \mbox{for}\ c_1>\frac{5}{4}- 4k^2,\\[5pt]
\alpha_1=\alpha_2,\ \alpha_3=\alpha_4,\ \frac{\alpha_1-\alpha_4}{4}=-k &\quad   \mbox{for}\ c_1=\frac{5}{4}- 4k^2,\\[5pt]
\frac{\alpha_2-\alpha_3}{4}=-k &\quad \mbox{for}\  c_1<\frac{5}{4}- 4k^2.
\end{align}
\item[(b)]  If $(c_1,c_2)\in \mathbb{P}_k$,   then 
\begin{align}
\frac{\alpha_1-\alpha_3}{4} =\frac{\alpha_2-\alpha_4}{4}=-k &\quad  \mbox{for}\ c_1>\frac{5}{4}- 8k^2,\\[5pt]
\alpha_2=\alpha_3,\ \frac{\alpha_1-\alpha_3}{4}= \frac{\alpha_2-\alpha_4}{4}=-k &\quad   \mbox{for}\ c_1=\frac{5}{4}- 8k^2,\\[5pt]
\frac{\alpha_1-\alpha_2}{4}= \frac{\alpha_3-\alpha_4}{4}=-k &\quad \mbox{for}\  c_1<\frac{5}{4}- 8k^2.
\end{align}
\end{itemize}
\end{lemma}
%%%%%%%%
\begin{proof}
Again, the proof is straightforward.
\end{proof}
%%%%%%%%

In the following, $G_{0,4}^{2,0}\left(\!\begin{array}{c}\\
{\alpha, \beta,\gamma,\delta}\end{array} \bigg\vert\, z \right)$ denotes a Meijer's $G$-function (see, e,g., \cite[Sects.~5.3--5.6]{EMOT53}, \cite[Ch.~16]{OLBC10}). 
 
 %%%%%%%%
\begin{theorem} \lb{tA.3} 
Let $(c_1,c_2)\in \bbR^2$.
\begin{itemize}
\item[(a)] If $(c_1,c_2)$ lies on exactly one  line, say $\mathbb{L}_k$,   and   on none of the parabolas, then a fundamental system of solutions of \eqref{A.1} is given by 
\begin{align}
y_1(\lambda;r)&=G_{0,4}^{2,0}\left(\!\begin{array}{c}\\
{\frac{\alpha_1}{4},\frac{\alpha_4}{4}, \frac{\alpha_2}{4},\frac{\alpha_3}{4}}
\end{array} \bigg\vert\,   \frac{\lambda r^4}{256} \right),\\
y_2(\lambda;r)&=r^{\alpha_2} \,_{0} F_3\left(\!\begin{array}{c}\\
{ 1+\frac{\alpha_2-\alpha_4}{4}, 1+\frac{\alpha_1-\alpha_3}{4}, 1+\frac{\alpha_2-\alpha_4}{4}}
\end{array} \bigg\vert\,  \frac{\lambda r^4}{256} \right),\\
y_3(\lambda;r)&=r^{\alpha_3} \,_{0} F_3\left(\!\begin{array}{c}\\
{ 1+\frac{\alpha_4-\alpha_1}{4}, 1+\frac{\alpha_4-\alpha_2}{4}, 1+\frac{\alpha_4-\alpha_3}{4}}
\end{array} \bigg\vert\,  \frac{\lambda r^4}{256} \right),\\
\quad y_4(\lambda;r)&=r^{\alpha_4} \,_{0} F_3\left(\!\begin{array}{c}\\
{ 1+\frac{\alpha_3-\alpha_1}{4}, 1+\frac{\alpha_3-\alpha_2}{4}, 1+\frac{\alpha_3-\alpha_4}{4}}
\end{array} \bigg\vert\,  \frac{\lambda r^4}{256} \right), 
\end{align}
if $\d c_1>\frac{5}{4}- 4k^2$, and by 
\begin{align}
y_1(\lambda;r)&=r^{\alpha_1} \,_{0} F_3\left(\!\begin{array}{c}\\
{ 1+\frac{\alpha_1-\alpha_2}{4}, 1+\frac{\alpha_1-\alpha_3}{4}, 1+\frac{\alpha_1-\alpha_4}{4}}
\end{array} \bigg\vert\,  \frac{\lambda r^4}{256} \right),\\
y_2(\lambda;r)&=G_{0,4}^{2,0}\left(\!\begin{array}{c}\\
{\frac{\alpha_2}{4},\frac{\alpha_3}{4}, \frac{\alpha_1}{4},\frac{\alpha_4}{4}}
\end{array} \bigg\vert\,   \frac{\lambda r^4}{256} \right),\\
y_3(\lambda;r)&=r^{\alpha_3} \,_{0} F_3\left(\!\begin{array}{c}\\
{ 1+\frac{\alpha_4-\alpha_1}{4}, 1+\frac{\alpha_4-\alpha_2}{4}, 1+\frac{\alpha_4-\alpha_3}{4}}
\end{array} \bigg\vert\,  \frac{\lambda r^4}{256} \right),\\
\quad y_4(\lambda;r)&=r^{\alpha_4} \,_{0} F_3\left(\!\begin{array}{c}\\
{ 1+\frac{\alpha_3-\alpha_1}{4}, 1+\frac{\alpha_3-\alpha_2}{4}, 1+\frac{\alpha_3-\alpha_4}{4}}
\end{array} \bigg\vert\,  \frac{\lambda r^4}{256} \right),  
\end{align}
if $\d c_1<\frac{5}{4}- 4k^2$.  \\[1mm] 
\item[(b)]
If $(c_1,c_2)$ lies on exactly one parabola, say $\mathbb{P}_k$, and on none of the lines, then a fundamental system of solutions of \eqref{A.1} is given by 
\begin{align}
y_1(\lambda;r)&=G_{0,4}^{2,0}\left(\!\begin{array}{c}\\
{\frac{\alpha_1}{4},\frac{\alpha_3}{4}, \frac{\alpha_2}{4},\frac{\alpha_4}{4}}
\end{array} \bigg\vert\,   \frac{\lambda r^4}{256} \right),\\
y_2(\lambda;r)&=G_{0,4}^{2,0}\left(\!\begin{array}{c}\\
{\frac{\alpha_2}{4},\frac{\alpha_4}{4}, \frac{\alpha_1}{4},\frac{\alpha_3}{4}}
\end{array} \bigg\vert\,   \frac{\lambda r^4}{256} \right),\\
y_3(\lambda;r)&=r^{\alpha_3} \,_{0} F_3\left(\!\begin{array}{c}\\
{ 1+\frac{\alpha_3-\alpha_1}{4}, 1+\frac{\alpha_3-\alpha_2}{4}, 1+\frac{\alpha_3-\alpha_4}{4}}
\end{array} \bigg\vert\,  \frac{\lambda r^4}{256} \right),\\
y_4(\lambda;r)&=r^{\alpha_4} \,_{0} F_3\left(\!\begin{array}{c}\\
{ 1+\frac{\alpha_4-\alpha_1}{4}, 1+\frac{\alpha_4-\alpha_2}{4}, 1+\frac{\alpha_4-\alpha_3}{4}}
\end{array} \bigg\vert\,  \frac{\lambda r^4}{256} \right), 
\end{align}
if $\d c_1>\frac{5}{4}- 8k^2$, and by 
\begin{align}
y_1(\lambda;r)&=G_{0,4}^{2,0}\left(\!\begin{array}{c}\\
{\frac{\alpha_1}{4},\frac{\alpha_2}{4}, \frac{\alpha_3}{4},\frac{\alpha_4}{4}}
\end{array} \bigg\vert\,   \frac{\lambda r^4}{256} \right),\\
y_2(\lambda;r)&=r^{\alpha_2} \,_{0} F_3\left(\!\begin{array}{c}\\
{ 1+\frac{\alpha_2-\alpha_1}{4}, 1+\frac{\alpha_2-\alpha_3}{4}, 1+\frac{\alpha_2-\alpha_4}{4}}
\end{array} \bigg\vert\,  \frac{\lambda r^4}{256} \right),\\
y_3(\lambda;r)&=G_{0,4}^{2,0}\left(\!\begin{array}{c}\\
{\frac{\alpha_3}{4},\frac{\alpha_4}{4}, \frac{\alpha_1}{4},\frac{\alpha_2}{4}}
\end{array} \bigg\vert\,   \frac{\lambda r^4}{256} \right),\\
y_4(\lambda;r)&=r^{\alpha_4} \,_{0} F_3\left(\!\begin{array}{c}\\
{ 1+\frac{\alpha_4-\alpha_1}{4}, 1+\frac{\alpha_4-\alpha_2}{4}, 1+\frac{\alpha_4-\alpha_3}{4}}
\end{array} \bigg\vert\,  \frac{\lambda r^4}{256} \right),
\end{align}
if $\d c_1<\frac{5}{4}- 8k^2$.     \\[1mm] 
\item[(c)] If $(c_1,c_2)$ lies on exactly two distinct lines and on none of the parabolas, then a fundamental system of solutions of \eqref{A.1} is given by 
\begin{align}
y_1(\lambda;r)&=G_{0,4}^{2,0}\left(\!\begin{array}{c}\\
{\frac{\alpha_1}{4},\frac{\alpha_4}{4}, \frac{\alpha_2}{4},\frac{\alpha_3}{4}}
\end{array} \bigg\vert\,   \frac{\lambda r^4}{256} \right),\\
y_2(\lambda;r)&=G_{0,4}^{2,0}\left(\!\begin{array}{c}\\
{\frac{\alpha_2}{4},\frac{\alpha_3}{4}, \frac{\alpha_1}{4},\frac{\alpha_4}{4}}
\end{array} \bigg\vert\,   \frac{\lambda r^4}{256} \right),\\
y_3(\lambda;r)&=r^{\alpha_3} \,_{0} F_3\left(\!\begin{array}{c}\\
{ 1+\frac{\alpha_3-\alpha_1}{4}, 1+\frac{\alpha_3-\alpha_2}{4}, 1+\frac{\alpha_3-\alpha_4}{4}}
\end{array} \bigg\vert\,  \frac{\lambda r^4}{256} \right),\\
y_4(\lambda;r)&=r^{\alpha_4} \,_{0} F_3\left(\!\begin{array}{c}\\
{ 1+\frac{\alpha_4-\alpha_1}{4}, 1+\frac{\alpha_4-\alpha_2}{4}, 1+\frac{\alpha_4-\alpha_3}{4}}
\end{array} \bigg\vert\,  \frac{\lambda r^4}{256} \right).
\end{align}
\item[(d)] If $(c_1,c_2)$ lies on at least one line and at least one parabola, then a fundamental system of solutions of \eqref{A.1} is given by 
\begin{align}
y_1(\lambda;r)&=G_{0,4}^{4,0}\left(\!\begin{array}{c}\\
{\frac{\alpha_1}{4},\frac{\alpha_2}{4}, \frac{\alpha_3}{4},\frac{\alpha_4}{4}}
\end{array} \bigg\vert\,   \frac{\lambda r^4}{256} \right),\\
y_2(\lambda;r)&=G_{0,4}^{3,0}\left(\!\begin{array}{c}\\
{\frac{\alpha_2}{4},\frac{\alpha_3}{4}, \frac{\alpha_4}{4},\frac{\alpha_1}{4}}
\end{array} \bigg\vert\,  - \frac{\lambda r^4}{256} \right),\\
y_3(\lambda;r)&=G_{0,4}^{2,0}\left(\!\begin{array}{c}\\
{\frac{\alpha_3}{4},\frac{\alpha_4}{4}, \frac{\alpha_1}{4},\frac{\alpha_2}{4}}
\end{array} \bigg\vert\,   \frac{\lambda r^4}{256} \right),\\
y_4(\lambda;r)&=r^{\alpha_4} \,_{0} F_3\left(\!\begin{array}{c}\\
{ 1+\frac{\alpha_4-\alpha_1}{4}, 1+\frac{\alpha_4-\alpha_2}{4}, 1+\frac{\alpha_4-\alpha_3}{4}}
\end{array} \bigg\vert\,  \frac{\lambda r^4}{256} \right).
\end{align}
\end{itemize}
\end{theorem}
%%%%%%%%
\begin{proof}
This follows from Lemmas \ref{lA.1} and \ref{lA.2} and our general result in \cite{GH23}, specifically, \cite[Theorem~4.3]{GH23} and its proof. 
\end{proof}
%%%%%%%%

%\newpage 

%%%%%%%%
\begin{figure}[H]
\includegraphics[width=1\textwidth]{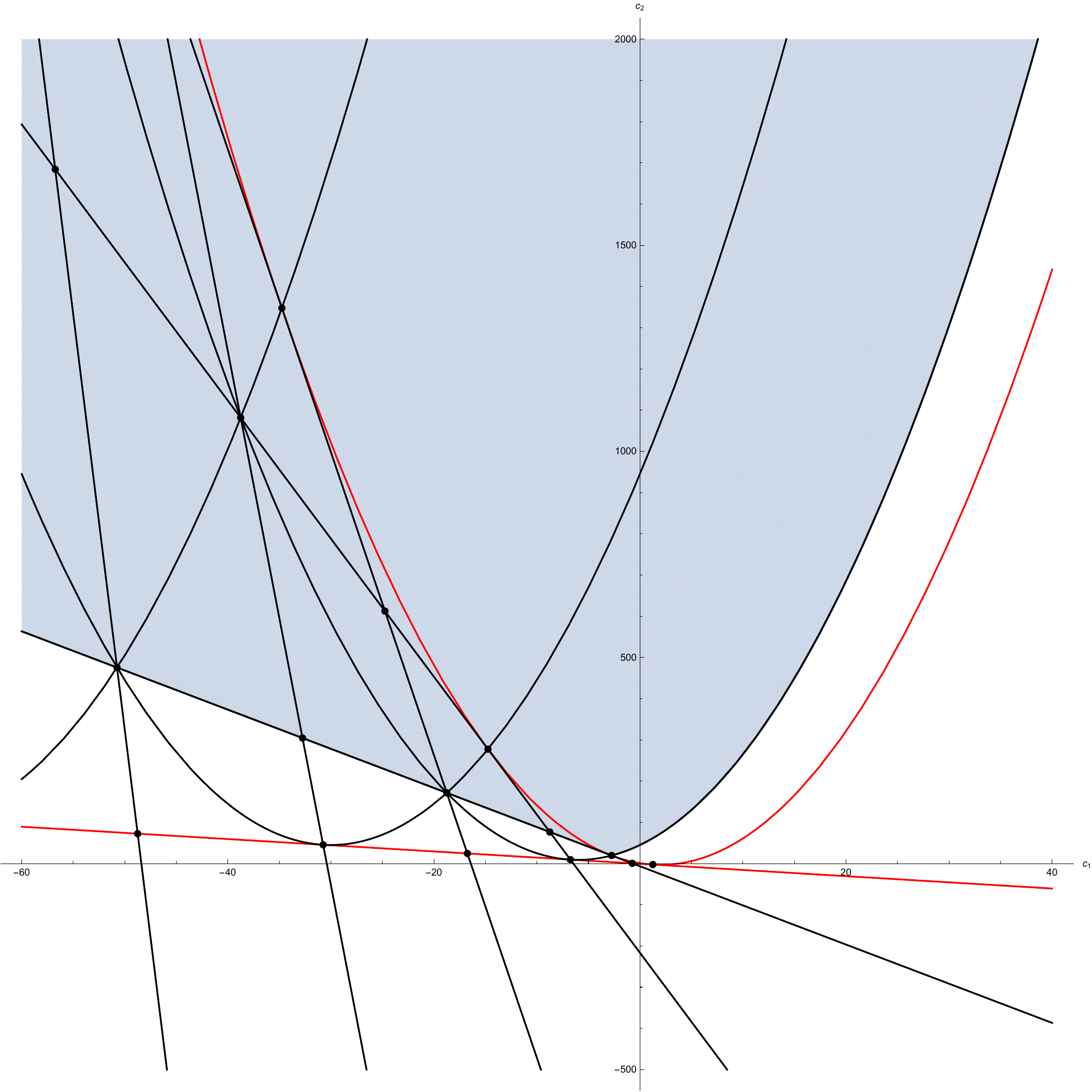}\qquad
\caption{The graph shows the  lines  $\mathbb{L}_h$, $h\in \bbN_0$, $0\leq h\leq 5$, and the parabolas $\mathbb{P}_k$, $k\in \bbN_0$, $0\leq k\leq 3$. The line $\mathbb{L}_0$ and the parabola $\mathbb{P}_0$
are shown in red. The shaded region (including its boundary) corresponds to the set of all $(c_1,c_2)\in \bbR^2$ such that the differential operator $\tau_2(c_1,c_2)|_{C_0^{\infty}((0,\infty))}$ is essentially self-adjoint in $L^2((0,\infty))$.}
\lb{F:parabolas and lines}
\end{figure}
%%%%%%%%

\medskip
%%%%%%%%%%%%%%%%%%%%%%%%%%%%%%%%%%%%%
\noindent 
{\bf Acknowledgments.} We gratefully acknowledge the very helpful comments provided by both referees. 
%%%%%%%%%%%%%%%%%%%%%%%%%%%%%%%%%%%%%

\bigskip 

$\bullet$ The authors declare that they have no conflict of interest. 

$\bullet$ No funding was received for conducting this particular research. 

$\bullet$ The authors have no relevant financial or non-financial interests to disclose.

$\bullet$ The authors have no competing interests to declare that are relevant to the content of this article.

$\bullet$ All authors certify that they have no affiliations with or involvement in any organization or entity with any financial interest or non-financial interest in the subject matter or materials discussed in this manuscript.

$\bullet$ The authors have no financial or proprietary interests in any material discussed in this article.

%%%%%%%%%%%%%%%%%%%%%%%%%%%%%%%%%%%%%%%

\end{document}